\documentclass[12pt,reqno]{amsart}

\usepackage{stmaryrd}
\usepackage[english]{babel}
\usepackage{afterpage}
\usepackage{caption}
\usepackage{a4wide}
\usepackage{pgf}
\usepackage{subfigure}
\usepackage{multirow}
\newcommand{\Oh}{\mathcal{T}_h}
\newcommand{\Eh}{\mathcal{E}_h}

\newcommand{\n}{\boldsymbol{n}}
\newcommand{\Piv}{\underline{\boldsymbol{\Pi_V}}}
\newcommand{\Piw}{\boldsymbol{\Pi_W}}

\newcommand{\Pim}{\boldsymbol{P_M}}
\newcommand{\eu}{\boldsymbol{e_u}}
\newcommand{\esigma}{\underline{\boldsymbol{e_{\sigma}}}}

\newcommand{\euhat}{\boldsymbol{e_{\widehat{u}}}}
\newcommand{\deltasigma}{\underline{\boldsymbol{\sigma}} - \Piv \underline{\boldsymbol{\sigma}}}
\newcommand{\deltau}{\boldsymbol{u} - \Piw \boldsymbol{u}}

\newcommand{\bint}[2]{( #1\,,\,#2 )_{\Oh}}
\newcommand{\bintEh}[2]{\langle #1\,,\,#2 \rangle_{\partial{\Oh}}}
\newcommand{\bintEhi}[2]{\langle #1\,,\,#2 \rangle_{\partial{\Oh} \backslash \partial \Omega}}
\newcommand{\bintEhb}[2]{\langle #1\,,\,#2 \rangle_{\partial{\Omega}}}
\newcommand{\bintK}[2]{\langle #1\,,\,#2 \rangle_{\partial{K}}}

\numberwithin{equation}{section}

\newtheorem{theorem}{Theorem}[section]
\newtheorem{lemma}[theorem]{Lemma}
\newtheorem{corollary}{Corollary}[section]
\newtheorem{proposition}{Proposition}[section]

\theoremstyle{definition}

\theoremstyle{remark}
\newtheorem{remark}[theorem]{Remark}
\newtheorem{assumption}{Assumption}[section]

\usepackage{color}
\definecolor{black}{rgb}{0,0,0}

\definecolor{red}{rgb}{1,0,0}

\definecolor{blue}{rgb}{0,0,1}

\title{An HDG method for linear elasticity with strong symmetric stresses}

\author{Weifeng Qiu}
\address{Department of Mathematics, City University of Hong Kong, 83 Tat Chee Avenue, Hong Kong} 
\email{weifeqiu@cityu.edu.hk}

\author{Jiguang Shen}
\address{Department of Mathematics, University of Minnesota, Minneapolis, MN 55455, USA} 
\email{shenx179@umn.edu}

\author{Ke Shi}
\address{Department of Mathematics \& Statistics, Old Dominion University, Norfolk, VA 23529, USA} 
\email{kshi@odu.edu}

\begin{document}

\begin{abstract}
This paper presents a new hybridizable discontinuous Galerkin (HDG) method for linear elasticity on 
general polyhedral meshes, based on a strong symmetric stress formulation. The key feature of this 
new HDG method is the use of a special form of the numerical trace of the stresses, which makes 
the error analysis different from the projection-based error analyzes used for most other 
HDG methods. For arbitrary polyhedral elements, we approximate the stress by using polynomials of
degree $k\ge 1$ and the displacement by using polynomials of degree $k+1$. In contrast, to approximate 
the numerical trace of the displacement on the faces, we use polynomials of degree $k$ only. This allows
for a very efficient implementation of the method, since the numerical trace of the displacement is the only 
globally-coupled unknown, but does not degrade the convergence properties of the method. Indeed,  
we prove optimal orders of convergence for both the stresses and displacements on the elements. In the almost 
incompressible case, we show the error of the stress is also optimal in the standard $L^2-$norm. 
These optimal results are possible thanks to a special superconvergence property of the numerical traces of 
the displacement, and thanks to the use of a crucial elementwise Korn's inequality. Several numerical results are
presented to support our theoretical findings in the end. 
\end{abstract}



\maketitle

\section{Introduction} 
In this paper, we introduce a new hybridizable discontinuous Galerkin (HDG) method for the system of linear elasticity
\begin{subequations}\label{elasticity_equation}
 \begin{alignat}{2}
  \label{elasticity_1}
\mathcal{A} \underline{\boldsymbol{\sigma}} - 
\underline{\boldsymbol{\epsilon}}(\boldsymbol{u}) & = 0 && \text{in $\Omega\subset\mathbb{R}^3$,}\\ 
\label{elasticity_2}
\nabla \cdot \underline{\boldsymbol{\sigma}} & = \boldsymbol{f} \quad && \text{in $\Omega$,}\\
\label{elasticity_3}
\boldsymbol{u} &= \boldsymbol{g} && \text{on $\partial \Omega$}.
 \end{alignat}
\end{subequations}
Here, the displacement is denoted by the vector field $\boldsymbol{u}:\Omega \rightarrow \mathbb{R}^{3}$.
The strain tensor is represented by $\underline{\boldsymbol{\epsilon}}(\boldsymbol{u}): = \frac{1}{2}(\nabla \boldsymbol{u} + (\nabla \boldsymbol{u})^{\top})$.
The stress tensor is represented by $\underline{\boldsymbol{\sigma}}: \Omega \rightarrow \boldsymbol{S}$, where 
$\boldsymbol{S}$ denotes the set of all symmetric matrices in $\mathbb{R}^{3\times 3}$.
The compliance tensor $\mathcal{A}$ is assumed to be a bounded, symmetric, positive definite tensor over $\boldsymbol{S}$.
The body force $\boldsymbol{f}$ lies in $\boldsymbol{L}^2(\Omega)$, the
displacement of the boundary $\boldsymbol{g}$ is a function in $\boldsymbol{H}^{1/2}(\partial\Omega)$ and $\Omega$ is a polyhedral domain.

In general, there are two approaches to design mixed finite element methods for linear elasticity.
The first approach is to enforce the symmetry of the stress tensor weakly (\cite{ArnoldBrezziDouglas84,ArnoldFalkWinther07,BoffiBrezziFortin09,
CockburnGopalakrishnanGuzman10,GopalakrishnanGuzman10,Guzman10,QiuDemko:2009:MME1,QiuDemko11,Stenberg88}). 
In this category, is included the HDG method considered in \cite{CockburnShi13_elas}.
The other approach is to exactly enforce the symmetry of the approximate stresses. 
The methods considered in
\cite{CockburnSchoetzauWang06,AdamsCockburn05,ArnoldAwanouWinther08,ArnoldAwanouWinther14,ArnoldWintherNC,
Awanou09,GopalakrishnanGuzman11,ManHuShi09,SoonCockburnStolarski09,Yi05,Yi06} 
belong to the second category, and so does the contribution of this paper. In general, the methods in the first category 
are easier to implement. On the other hand, the methods in the second category preserve the balance of angular momentum 
strongly and have less degrees of freedom.
Next, we compare our HDG method with several methods of the second category.

In \cite{CockburnSchoetzauWang06}, an LDG method using strongly symmetric
stresses (for isotropic linear elasticity) was introduced and proved to yield
convergence properties that remain unchanged when the material becomes
incompressible; simplexes and polynomial approximations or degree $k$ in all variables were
used. However, as all LDG methods for second-order elliptic problems,
although the displacement converges with order $k+1$, the strain and pressure converge
sub-optimally with order $k$. Also, the method cannot be hybridized. 
Stress finite elements satisfying both strong symmetry and $H(\text{div})$-conformity are introduced in 
\cite{AdamsCockburn05,ArnoldAwanouWinther08}. The main drawback of these methods is that 
they have too many degrees of freedom of stress elements and hybridization is not available for them 
(see detailed description in \cite{Guzman10}). In \cite{ArnoldAwanouWinther14,ArnoldWintherNC,
Awanou09,GopalakrishnanGuzman11,ManHuShi09,SoonCockburnStolarski09,Yi05,Yi06}, 
non-conforming methods using symmetric stress elements are introduced. But, 
methods in \cite{ArnoldAwanouWinther14,ArnoldWintherNC,
Awanou09,ManHuShi09,Yi05,Yi06} use low order finite element spaces only 
(most of them are restricted to rectangular or cubical meshes except \cite{ArnoldAwanouWinther14,ArnoldWintherNC}).
In \cite{GopalakrishnanGuzman11}, a family of simplicial elements (one for each
$k \geq 1$) are developed in both two and three dimensions.
(The degrees of freedom of $P_{k+1}(\boldsymbol{S},K)$ were studied in
\cite{GopalakrishnanGuzman11} and then used
to design the projection operator $\Pi^{(\text{div},\boldsymbol{S})}$ in \cite{GopaQiu:PracticalDPG}). 
However, the convergence rate of stress is suboptimal. The first HDG method for
linear and nonlinear elasticity was introduced in
\cite{SoonThesis08,SoonCockburnStolarski09}; see also the related HDG method
proposed in \cite{NguyenPeraire2012}. These methods also use simplexes and
polynomial approximations of degree $k$ in all variables. For general polyhedral
elements, this method was
recently analyzed in \cite{FuCockburnStolarski} where it was shown that the method converges optimally in the
displacement with order $k+1$, but with the suboptimal order of $k+1/2$ for the
pressure and the stress. For $k=1$, these orders of convergence were numerically
shown to be sharp for triangular elements. In this paper, we prove that
by enriching the local stress space to be polynomials of degree no more than
$k+1$, and by using a modified numerical trace, we are able to obtain optimal order of convergence for all unknowns. 
In addition, this analysis is valid for general polyhedral meshes. To the best of our knowledge, this is so far the only result 
which has optimal accuracy with general polyhedral triangulations for linear elasticity problems.

Like many hybrid methods, our HDG method provides approximation to stress and displacement in each element and trace of displacement along interfaces of meshes. In general, the corresponding finite element spaces are $\underline{\boldsymbol{V}}_h,\boldsymbol{W}_h,\boldsymbol{M}_h$, which are defined to be
\begin{alignat*}{3}
\underline{\boldsymbol{V}}_h=&\;\{\underline{\boldsymbol{v}} \in \underline{\boldsymbol{L}}^2(\Omega):
&&\quad\underline{\boldsymbol{v}}|_K\in \underline{\boldsymbol{V}}(K)  
&&\quad\forall\; K\in\Oh\},
\\
\boldsymbol{W}_h=&\;\{\boldsymbol{\omega}\in \boldsymbol{L}^2(\Omega):
&&\quad \boldsymbol{\omega}|_K \in \boldsymbol{W}(K)
&&\quad\forall\; K\in\Oh\},
\\
\boldsymbol{M}_h=&\;\{\boldsymbol{\mu}\in \boldsymbol{L}^2(\Eh):
&&\quad \boldsymbol{\mu}|_F\in \boldsymbol{M}(F) 
&&\quad\forall\; F\in\Eh\}.
\end{alignat*}
Here $\mathcal{T}_h$ denotes a triangulation of the domain $\Omega$ and $\Eh$ is the set of all faces $F$ of all elements $K \in \mathcal{T}_h$. The spaces
$\underline{\boldsymbol{V}}(K), \boldsymbol{W}(K), \boldsymbol{M}(F)$ are called the {\em local spaces} which are defined on each element/face. In Table \ref{table:error-order} we list several choices of local spaces for different methods. In this paper, our choice of the local spaces is defined as:
\[
\underline{\boldsymbol{V}}(K) = \underline{\boldsymbol{P}}_k(\boldsymbol{S}, K), \quad  \boldsymbol{W}(K) = \boldsymbol{P}_{k+1}(K), \quad  \boldsymbol{M}(F) =\boldsymbol{P}_k(F).
\]
Here, the space of vector-valued functions defined on $D$  whose entries
are polynomials of total degree $k$ is denoted by $\boldsymbol{P}_k(D)$ ($k\geq 1$). Similarly,
$\underline{\boldsymbol{P}}_k(\boldsymbol{S}, K)$ denotes the space of
symmetric-valued functions defined on $K$ whose entries are polynomials of total degree $k$. 
In addition, our method allows $\mathcal{T}_{h}$ to be any conforming polyhedral triangulation of $\Omega$.

Note the fact that the only globally-coupled degrees of freedom are those of the numerical
trace of displacement along $\Eh$, renders the method efficiently
implementable. However, the fact that the polynomial degree of the approximate
numerical traces of the displacement is one {\em less} than that of the
approximate displacement inside the elements, might cause a degradation in the
approximation properties of the displacement. However, this unpleasant situation
is avoided altogether by taking a special form
of the numerical trace of the stresses inspired on the choice taken in
\cite{Lehrenfeld10} in the framework of diffusion problems. This choice allows
for a special superconvergence of part of the numerical traces of the stresses
which, in turn, guarantees that, for $k\ge1$, 
the $L^{2}$-order of convergence for the stress is $k+1$ and that of the displacement $k+2$.
So, we obtain optimal convergence for both stress and displacement for general
polyhedral elements.
Let us mention that the approach of error analysis of our HDG method is 
different from the traditional projection-based error analysis in
\cite{CockburnGopalakrishnanSayas10,CockburnQiuShi11,CockburnShi13_elas} in
three aspects.
First, here, we use simple $L^{2}$-projections, not the numerical trace-tailored
projections typically used for the analysis of other HDG methods. Second, we
take the stabilization parameter to be of order $1/h$ instead of of order
one. And finally, we use an elementwise Korn's inequality
(Lemma~\ref{symmetric_grad}) to deal with the symmetry of the stresses.  

We notice that mixed methods in \cite{CockburnGopalakrishnanGuzman10,GopalakrishnanGuzman10} and 
HDG methods in \cite{CockburnShi13_elas} also achieve optimal convergence for stress and
superconvergence for displacement by post processing. However, there are two disadvantages regarding of implementation.
First, these methods enforce the stress symmetry weakly, which means that they have a much larger space for the stress.
In additon, these methods usually need to add matrix bubble functions ($\underline{\boldsymbol{\delta V}}$ in  \cite{CockburnGopalakrishnanGuzman10}) into their stress elements in order to obtain optimal approximations. 
In fact, the construction of such bubbles on general polyhedral elements is still an open problem.
In contrast, our method avoids using matrix bubble functions but only use simple polynomial space of degree $k, k+1$. In Table~\ref{table:error-order}, we compare methods which use $\boldsymbol{M}_{h}$ for approximating  trace of displacement $\widehat{\boldsymbol{u}}_{h}$ on $\Eh$. There, $\boldsymbol{u}^{\star}_h$ is a post-processed numerical solution of
displacement.

\begin{table}[ht]
\caption{Orders of convergence for methods for which {$\widehat{\boldsymbol{u}}_{h}\in \boldsymbol{M}(F)=\boldsymbol{P}_k(F), k\ge1,$} 
and $K$ is a tetrahedron.}
\centering
\begin{tabular}{c c c c c c c}
\hline
\noalign{\smallskip}
method             &$ \underline{\boldsymbol{V}}(K)$ 
&\hskip-.5truecm$\boldsymbol{W}(K)$ 
&\hskip-.3truecm$\|\underline{\boldsymbol{\sigma}} - \underline{\boldsymbol{\sigma}}_h\|_{\Oh}$
&\hskip-.3truecm$\|\ \boldsymbol{u} - \boldsymbol{u}_h\|_{\Oh}$
&\hskip-.3truecm$\|\boldsymbol{u} - \boldsymbol{u}^{\star}_h\|_{\Oh}$ \\
\noalign{\smallskip}
\hline\hline
\noalign{\smallskip}
AFW\cite{ArnoldFalkWinther07} &  $\underline{\boldsymbol{P}}_{k}(\mathbb{R}^{3\times 3}, K)$ &\hskip-.2truecm  
$\boldsymbol{P}_{k-1}(K)$&     \hskip-.5truecm   $k$          &\hskip-.5truecm $k$  & \hskip-.5truecm - \\
CGG\cite{CockburnGopalakrishnanGuzman10}&  $\underline{\mathbf{RT}}_{k}(K)+\underline{\boldsymbol{\delta V}}$ &\hskip-.2truecm  $\boldsymbol{P}_{k}(K)$&     \hskip-.5truecm   $k+1$          &\hskip-.5truecm $k+1$  & \hskip-.5truecm$k+2$             \\
GG\cite{GopalakrishnanGuzman10}   &  $\underline{\boldsymbol{P}}_{k}(\mathbb{R}^{3\times 3},K)+\underline{\boldsymbol{\delta V}}$ &\hskip-.2truecm  $\boldsymbol{P}_{k-1}(K)$  &    \hskip-.5truecm   $k+1$          &\hskip-.5truecm $k$   &\hskip-.5truecm $k+1$            \\
CS\cite{CockburnShi13_elas}     &   $\underline{\boldsymbol{P}}_{k}(\mathbb{R}^{3\times 3},K)+\underline{\boldsymbol{\delta V}}$ &\hskip-.2truecm  $\boldsymbol{P}_{k}(K)$ &      \hskip-.5truecm   $k+1$          &\hskip-.5truecm $k+1$   &\hskip-.5truecm $k+2$               \\
GG\cite{GopalakrishnanGuzman11}     &  $\underline{\boldsymbol{P}}_{k+1}(\boldsymbol{S},K)$&\hskip-.2truecm $\boldsymbol{P}_{k}(K)$&        \hskip-.5truecm $k$         & \hskip-.5truecm$k+1$ &\hskip-.5truecm -                \\
HDG-S    &  $\underline{\boldsymbol{P}}_{k}(\boldsymbol{S},K)$ &\hskip-.2truecm $\boldsymbol{P}_{k+1}(K)$ &    \hskip-.5truecm     $k+1$          &\hskip-.5truecm $k+2$ &\hskip-.5truecm -
\\
\noalign{\smallskip}
\hline
\end{tabular}
\label{table:error-order}  
\end{table}

The remainder of this paper is organized as follows. In Section $2$, we introduce our HDG method and present our a priori error estimates.
In Section $3$, we give a characterization of the HDG method and show the global matrix is symmetric and positive definite.
In Section $4$, we give elementwise Korn's inequality in Lemma~\ref{symmetric_grad}, then provide a detailed proof of the a priori error estimates.
In Section $5$, we present several numerical examples in order to illustrate and test our method.
\section{Main results} In this section we first present the method in details and then show the main results for the error estimates.

\subsection{The HDG formulation with strong symmetry} Let us begin by introducing some notations and conventions. 
We adapt to our setting the notation used in \cite{CockburnQiuShi11}. Let 
$\mathcal{T}_h$ denote a conforming triangulation of $\Omega$
made of shape-regular polyhedral elements $K$. We recall that $\partial \mathcal{T}_h :=
\{\partial K : K \in \mathcal{T}_h \}$, and $\mathcal{E}_h$ 
denotes the set of all faces $F$ of all elements. We
denote by $\mathcal{F}(K)$ the set of all faces $F$ of the element $K$. We also use the standard 
notation to denote scalar, vector and tensor spaces. Thus, if $D(K)$ denotes
a space of scalar-valued functions defined on $K$, the corresponding space of
vector-valued functions is
$\boldsymbol{D}(K) := [D(K)]^d$ and the corresponding space of matrix-valued
functions is 
$\underline{\boldsymbol{D}}(K) := [D(K)]^{d\times d}$. Finally, $\underline{\boldsymbol{D}}(\boldsymbol{S}, K)$ denotes the symmetric  subspace of $\underline{\boldsymbol{D}}(K)$.

The methods we consider seek an approximation $(\underline{\boldsymbol{\sigma}}_h, \boldsymbol{u}_h, \widehat{\boldsymbol{u}}_h)$
to the exact solution $(\underline{\boldsymbol{\sigma}}, \boldsymbol{u},
 \boldsymbol{u}|_{\Eh})$ in the finite dimensional space 
$\underline{\boldsymbol{V}}_h \times \boldsymbol{W}_h \times \boldsymbol{M}_h \subset \underline{\boldsymbol{L}}^2(\boldsymbol{S}, \Omega)
\times \boldsymbol{L}^2(\Omega) \times \boldsymbol{L}^2(\Eh)$ given by
\begin{subequations}
\label{spaces}
\begin{alignat}{3}
\label{spaces-V}
\underline{\boldsymbol{V}}_h=&\;\{\underline{\boldsymbol{v}} \in \underline{\boldsymbol{L}}^2(\boldsymbol{S}, \Omega):
&&\quad\underline{\boldsymbol{v}}|_K\in \underline{\boldsymbol{P}}_k(\boldsymbol{S}, K)
&&\quad\forall\; K\in\Oh\},
\\
\label{spaces-W}
\boldsymbol{W}_h=&\;\{\boldsymbol{\omega}\in \boldsymbol{L}^2(\Omega):
&&\quad \boldsymbol{\omega}|_K \in \boldsymbol{P}_{k+1}(K)
&&\quad\forall\; K\in\Oh\},
\\
\label{spaces-M}
\boldsymbol{M}_h=&\;\{\boldsymbol{\mu}\in \boldsymbol{L}^2(\Eh):
&&\quad \boldsymbol{\mu}|_F\in \boldsymbol{P}_k(F)
&&\quad\forall\; F\in\Eh\}.
\end{alignat}
\end{subequations}
Here $P_k(D)$ denotes the standard space of polynomials of degree no more than $k$ on $D$. Here we require $k \ge 1$. 

The numerical approximation $(\underline{\boldsymbol{\sigma}}_h, \boldsymbol{u}_h, \widehat{\boldsymbol{u}}_h)$ can now be defined as the solution of the following system:
\begin{subequations}\label{HDG_formulation}
\begin{alignat}{1}
\bint{\mathcal{A} \underline{\boldsymbol{\sigma}}_h}{\underline{\boldsymbol{v}}} + \bint{\boldsymbol{u}_h}{\nabla \cdot \underline{\boldsymbol{v}}}
- \bintEh{\widehat{\boldsymbol{u}}_h}{\underline{\boldsymbol{v}} \n} &= 0, \\
\bint{\underline{\boldsymbol{\sigma}}_h}{\nabla \boldsymbol{\omega}} - \bintEh{\widehat{\underline{\boldsymbol{\sigma}}}_h \n}{\boldsymbol{\omega}} & = -\bint{\boldsymbol{f}}{\boldsymbol{\omega}},\\
\label{strong_cont}
\bintEhi{\widehat{\underline{\boldsymbol{\sigma}}}_h \n}{\boldsymbol{\mu}} & = 0,\\
\bintEhb{\widehat{\boldsymbol{u}}_h}{\boldsymbol{\mu}} & = \bintEhb{{\boldsymbol{g}}}{\boldsymbol{\mu}},
\intertext{for all $(\underline{\boldsymbol{v}}, \boldsymbol{\omega}, \boldsymbol{\mu}) \in
\underline{\boldsymbol{V}}_h \times \boldsymbol{W}_h \times
\boldsymbol{M}_h$, where}
\label{numerical_trace}
\widehat{\underline{\boldsymbol{\sigma}}}_h \n = \underline{\boldsymbol{\sigma}}_h \n - \tau (\Pim \boldsymbol{u}_h - \widehat{\boldsymbol{u}}_h) \quad \text{on $\partial \Oh$.}
\end{alignat} 
\end{subequations}
In fact, in Christoph Lehrenfeld's thesis, the author defines the numerical flux in this way for diffusion problems (see Remark $1.2.4$ in \cite{Lehrenfeld10}). 
This method was then analyzed for diffusion recently in \cite{Oikawa2015}.
Here, $\Pim$ denotes the standard $L^2$-orthogonal projection from $\boldsymbol{L}^2(\Eh)$ onto $\boldsymbol{M}_h$. We write
 $\bint{\underline{\boldsymbol{\eta}}}{\underline{\boldsymbol{\zeta}}}$ $ : = \sum^n_{i,j = 1} \bint{\underline{\boldsymbol{\eta}}_{i,j}}{\underline{\boldsymbol{\zeta}}_{i,j}}$,  $\bint{\boldsymbol{\eta}}{\boldsymbol{\zeta}} : = \sum^n_{i = 1} \bint{\eta_i}{\zeta_i}$, and
$
\bint{\eta}{\zeta} := \sum_{K \in \Oh} (\eta, \zeta)_K,
$
where $(\eta,\zeta)_D$ denotes the integral of $\eta\zeta$ over $D \subset \mathbb{R}^n$. Similarly, we write $\bintEh{\boldsymbol{\eta}}{\boldsymbol{\zeta}}:= \sum^n_{i=1} \bintEh{\eta_i}{\zeta_i}$
and $\bintEh{\eta}{\zeta}:= \sum_{K \in \Oh} \langle \eta \,,\,\zeta \rangle_{\partial K}$,
where $\langle \eta \,,\,\zeta \rangle_{D}$ denotes the integral of $\eta \zeta$ over $D \subset \mathbb{R}^{n-1}$. 

The parameter $\tau$ in \eqref{numerical_trace} is called the {\em{stabilization parameter}}. In this paper, we assume it is a fixed positive number on all faces. It is worth to mention that the numerical trace \eqref{numerical_trace} is defined slightly different from the usual {{HDG}} 
setting, see \cite{CockburnQiuShi11}. Namely, in the definition, we use $\Pim \boldsymbol{u}_h$ instead of $\boldsymbol{u}_h$. Indeed, this is a crucial modification in order to 
get error estimate. An intuitive explanation is that we want to preserve the strong continuity of the flux across the interfaces. Without the projection $\Pim$, by \eqref{strong_cont} the normal component of $\widehat{\underline{\boldsymbol{\sigma}}}_h$ is only weakly continuous across the interfaces.

\subsection{A priori error estimates}
To state our main result, we need to introduce some notations.
We define 
\begin{equation*}
\Vert \underline{\boldsymbol{v}}\Vert_{L^{2}(\mathcal{A},\Omega)} = \sqrt{(\mathcal{A}\underline{\boldsymbol{v}}, \underline{\boldsymbol{v}})_{\Omega}}, 
\quad \forall \underline{\boldsymbol{v}} \, \in \underline{\boldsymbol{L}}^{2}(\boldsymbol{S},\Omega).
\end{equation*}
We use $\|\cdot\|_{s, D}, |\cdot|_{s, D}$ to denote the usual norm and semi-norm on
the Sobolev space $H^s(D)$. 
We discard the first index $s$ if $s=0$. A differential operator with a sub-index $h$ means it is defined on each element $K \in \Oh$. Similarly, the norm $\|\cdot\|_{s, \Oh}$ is the discrete norm defined as $\|\cdot\|_{s, \Oh}:= \sum_{K \in \Oh}\|\cdot\|_{s, K}$.
Finally, we need an elliptic regularity assumption stated as follows.
Let $(\boldsymbol{\phi}, \underline{\boldsymbol{\psi}}) \in \boldsymbol{H}^2(\Omega) \times \underline{\boldsymbol{H}}^1(\Omega)$ be the solution of the adjoint problem:
\begin{subequations}\label{dual_problem}
\begin{alignat}{2}
\label{dual_problem_1}
\mathcal{A}\underline{\boldsymbol{\psi}} - \underline{\epsilon} (\boldsymbol{\phi}) & = 0 && \text{in $\Omega$},\\
\label{dual_problem_2}
\nabla \cdot \underline{\boldsymbol{\psi}} & = \eu \quad && \text{in $\Omega$},\\
\label{dual_problem_3}
\boldsymbol{\phi} & = 0 && \text{on $\partial \Omega$}.
\end{alignat}
\end{subequations}
We assume the solution
$(\boldsymbol{\phi}, \underline{\boldsymbol{\psi}})$ has the following 
elliptic regularity property:
\begin{equation}\label{regularity}
\|\underline{\boldsymbol{\psi}}\|_{1, \Omega} + \|\boldsymbol{\phi}\|_{2, \Omega} \le C_{reg} \|\eu\|_{\Omega},
\end{equation}
The assumption holds in the case of planar elasticity with scalar coefficients on
a convex domain, see \cite{BacutaBramble03}.

We are now ready to state our main result.

\begin{theorem}\label{Main_result}
If the meshes are quasi-uniform and $\tau = \mathcal{O}(\frac{1}{h})$, 
then we have
\begin{equation}\label{error_estimate_sigma}
\|\underline{\boldsymbol{\sigma}} - \underline{\boldsymbol{\sigma}}_h\|_{L^{2}(\mathcal{A},\Omega)} \le C h^s (\|\boldsymbol{u}\|_{s+1, \Omega} + \|\underline{\boldsymbol{\sigma}}\|_{s, \Omega}),
\end{equation}  
for all $1 \le s \le k+1$. Moreover, if the elliptic regularity property \eqref{regularity} holds, then we have
\begin{equation}
\label{error_estimate_u} \| \boldsymbol{u} - \boldsymbol{u}_h\|_{\Omega} \le C h^{s+1}(\|\boldsymbol{u}\|_{s+1, \Omega} + \|\underline{\boldsymbol{\sigma}}\|_{s, \Omega}), 
\end{equation}
for all $1 \le s \le k+1$. Here the constant $C$ depends on the upper bound of compliance tensor $\mathcal{A}$ 
but it is independent of the mesh size $h$.
\end{theorem}

This result shows that the numerical errors for both unknowns $(\boldsymbol{u}, \underline{\boldsymbol{\sigma}})$ are optimal. 
In addition, since the only globally-coupled unknown, $\widehat{\boldsymbol{u}}_{h}$,
stays in $\boldsymbol{P}_{k}(\Eh)$, the order of convergence for the
displacement remains optimal only because of a key superconvergence
property, see the remark right after Corollary \ref{estimate_sigma}.
In addition, we restrict our result on quasi-uniform meshes to make the proof simple and clear. 
This result holds for shape-regular meshes also.

\subsection{Numerical approximation for nearly incompressible materials}
Here, we consider the numerical approximation of stress for isotropic nearly incompressible materials. 

We define isotropic materials to be those whose compliance tensor satisfying the following Assumption~\ref{ass_isotropic}.
\begin{assumption}
\label{ass_isotropic}
\begin{align}
\label{ass_isotropic_def}
 \mathcal{A}\underline{\boldsymbol{\tau}} &= P_{D}\underline{\boldsymbol{\tau}}_{D} 
+ P_{T}\frac{\text{tr}(\underline{\boldsymbol{\tau}})}{3}I_{3}\\
\nonumber
\text{where } \quad  \underline{\boldsymbol{\tau}}_{D} & = \underline{\boldsymbol{\tau}} 
-\frac{\text{tr}(\underline{\boldsymbol{\tau}})}{3}I_{3}, 
\end{align}
for any $\underline{\boldsymbol{\tau}}$ in $\mathbb{R}^{3\times 3}$, and $P_{D}$ and $P_{T}$ are two positive constants. An isotropic material is nearly incompressible if $P_{T}$ is close to zero.

\end{assumption}

\begin{theorem}
\label{thm_locking_free}
If the material is isotropic (whose compliance tensor satisfies Assumption~\ref{ass_isotropic}), $P_{T}$ is positive, 
the boundary data $\boldsymbol{g}=0$, the meshes are quasi-uniform and $\tau = \mathcal{O}(\frac{1}{h})$, then we have 
\begin{equation}
\label{conv_locking_free}
\|\underline{\boldsymbol{\sigma}} - \underline{\boldsymbol{\sigma}}_h\|_{L^{2}(\Omega)}
 \le C h^{s} (\|\boldsymbol{u}\|_{s+1, \Omega} + \|\underline{\boldsymbol{\sigma}}\|_{s, \Omega}),
\end{equation}
for all $1 \le s \le k+1$. Here, the constant $C$ is independent of $P_{T}^{-1}$.
\end{theorem}

This result shows that the HDG method (\ref{HDG_formulation}) is locking-free for nearly incompressible materials. 
We emphasize that the convergence rate of stress for nearly incompressible materials is one order higher than   \cite{ArnoldFalkWinther07,GopalakrishnanGuzman11} with the same finite element space for numerical trace of displacement.

\section{A characterization of the HDG method}
\label{sec:hybrid}

In this section we show how to eliminate elementwise 
the unknowns $\underline{\boldsymbol{\sigma}}_h$ and $\boldsymbol{u}_h$ from 
the equations \eqref{HDG_formulation} and rewrite
the original system solely in terms of the unknown $\widehat{\boldsymbol{u}}_{h}$, 
see also \cite{SoonCockburnStolarski09}. Via this elimination, 
we do not have to deal with the large indefinite linear system generated by \eqref{HDG_formulation}, 
but with the inversion of a sparser symmetric positive definite matrix of remarkably smaller size.

\subsection{The local problems}
The result on the above mentioned elimination can be described using additional ``local" operators defined as follows:

On each element $K$, for any $\boldsymbol{\lambda}\in \boldsymbol{M}_h |_{\partial K}$, we denote $(\underline{\boldsymbol{Q}} \boldsymbol{\lambda}, \boldsymbol{U} \boldsymbol{\lambda}) \in \underline{\boldsymbol{V}}(K) \times \boldsymbol{W}(K)$ to be the unique solution of the local problem:
\begin{subequations}
\label{local_solvers}
\begin{align}
\label{local_solvers_eq1}
 (\mathcal{A} \underline{\boldsymbol{Q}}\boldsymbol{\lambda},  \underline{\boldsymbol{v}})_{K}
+(\boldsymbol{U}\boldsymbol{\lambda}, \nabla\cdot \underline{\boldsymbol{v}})_{K} 
&= \langle \boldsymbol{\lambda}, \underline{\boldsymbol{v}}\cdot \boldsymbol{n}\rangle_{\partial K},\\
\label{local_solvers_eq2}
 - (\nabla\cdot \underline{\boldsymbol{Q}}\boldsymbol{\lambda}, \boldsymbol{\omega})_{K} 
+ \langle \tau \Pim \boldsymbol{U}\boldsymbol{\lambda}, \boldsymbol{\omega}\rangle_{\partial K} 
&= \langle \tau \boldsymbol{\lambda}, \boldsymbol{\omega}\rangle_{\partial K},
\end{align}
\end{subequations}
for all $(\underline{\boldsymbol{v}}, \boldsymbol{\omega})\in \underline{\boldsymbol{V}}(K)\times \boldsymbol{W}(K)$.

On each element $K$, we also denote $(\underline{\boldsymbol{Q}}_S \boldsymbol{\lambda}, \boldsymbol{U}_S \boldsymbol{\lambda}) \in \underline{\boldsymbol{V}}(K) \times \boldsymbol{W}(K)$ to be the unique solution of the local problem:
\begin{subequations}
\label{local_solvers_source}
\begin{align}
\label{local_solvers_source_eq1}
 (\mathcal{A} \underline{\boldsymbol{Q}}_{S}\boldsymbol{f},  \underline{\boldsymbol{v}})_{K}
+(\boldsymbol{U}_{S}\boldsymbol{f}, \nabla\cdot \underline{\boldsymbol{v}})_{K} 
&= 0,\\
\label{local_solvers_source_eq2}
 - (\nabla\cdot \underline{\boldsymbol{Q}}_{S}\boldsymbol{f}, \boldsymbol{\omega})_{K} 
+ \langle \tau \Pim \boldsymbol{U}_{S}\boldsymbol{f}, \boldsymbol{\omega}\rangle_{\partial K} 
&= - (\boldsymbol{f}, \boldsymbol{\omega})_{K},
\end{align}
\end{subequations}
for all $(\underline{\boldsymbol{v}}, \boldsymbol{\omega})\in \underline{\boldsymbol{V}}(K)\times \boldsymbol{W}(K)$.

It is easy to show the two local problems are well-posted. In addition, due to the linearity of the global system \eqref{HDG_formulation},the numerical solution $(\underline{\boldsymbol{\sigma}}_h, \boldsymbol{u}_h, \widehat{\boldsymbol{u}}_h)$ satisfies
\begin{equation}
\label{local_to_global}
\underline{\boldsymbol{\sigma}}_h =  \underline{\boldsymbol{Q}}\widehat{\boldsymbol{u}}_h
+\underline{\boldsymbol{Q}}_{S}\boldsymbol{f}, \quad 
\boldsymbol{u}_h = \boldsymbol{U}\widehat{\boldsymbol{u}}_h+\boldsymbol{U}_{S}\boldsymbol{f}.
\end{equation}

\subsection{The global problem}
For the sake of simplicity, we assume the boundary data $\boldsymbol{g} = 0$. Then, the HDG method 
(\ref{HDG_formulation}) is to find $(\underline{\boldsymbol{\sigma}}_h, \boldsymbol{u}_h, \widehat{\boldsymbol{u}}_h)
\in \underline{\boldsymbol{V}}_h \times \boldsymbol{W}_h \times\boldsymbol{M}^0_h$ satisfying
\begin{subequations}
\label{HDG_formulation_zero}
\begin{alignat}{1}
\label{strong_cont_zero}
\bint{\mathcal{A} \underline{\boldsymbol{\sigma}}_h}{\underline{\boldsymbol{v}}} + \bint{\boldsymbol{u}_h}{\nabla \cdot \underline{\boldsymbol{v}}}-\bintEh{\widehat{\boldsymbol{u}}_h}{\underline{\boldsymbol{v}} \n} &= 0, \\
-\bint{\nabla\cdot\underline{\boldsymbol{\sigma}}_h}{\boldsymbol{\omega}} 
+ \bintEh{\tau (\Pim \boldsymbol{u}_h - \widehat{\boldsymbol{u}}_h)}{\boldsymbol{\omega}} 
& = -\bint{\boldsymbol{f}}{\boldsymbol{\omega}},\\
\label{HDG_formulation_zero_c}
\bintEhi{\underline{\boldsymbol{\sigma}}_h \n - \tau (\Pim \boldsymbol{u}_h - \widehat{\boldsymbol{u}}_h)}{\boldsymbol{\mu}} & = 0,
\end{alignat} 
\end{subequations}
for all $(\underline{\boldsymbol{v}}, \boldsymbol{\omega}, \boldsymbol{\mu}) \in
\underline{\boldsymbol{V}}_h \times \boldsymbol{W}_h \times
\boldsymbol{M}^0_h$, where $\boldsymbol{M}^0_{h} = 
\{\boldsymbol{\mu}\in \boldsymbol{M}_{h}: \boldsymbol{\mu}|_{\partial \Omega}=0 \}$.

Combining (\ref{HDG_formulation_zero_c}) with (\ref{local_to_global}), 
we have that for all $\boldsymbol{\mu}\in \boldsymbol{M}^0_{h}$,
\begin{equation}
\label{reduced_system_primitive}
\langle  (\underline{\boldsymbol{Q}}\widehat{\boldsymbol{u}}_h) \n
 - \tau (\Pim  \boldsymbol{U}\widehat{\boldsymbol{u}}_h 
 - \widehat{\boldsymbol{u}}_h ), \boldsymbol{\mu}\rangle_{\partial \mathcal{T}_{h}} 
= \langle  (\underline{\boldsymbol{Q}}_{S}\boldsymbol{f}) \n
 - \tau \Pim  \boldsymbol{U}_{S}\boldsymbol{f}, \boldsymbol{\mu}\rangle_{\partial \mathcal{T}_{h}}. 
\end{equation}
Up to now we can see that we only need to solve the reduced global linear system (\ref{reduced_system_primitive}) first, then 
recover $(\underline{\boldsymbol{\sigma}}_h, \boldsymbol{u}_h)$ by (\ref{local_to_global}) element by element. Next we show that the 
global system \eqref{reduced_system_primitive} is in fact symmetric positive definite.

\subsection{A characterization of the approximate solution}
The above results suggest the following characterization of the numerical solution of the HDG method. 

\begin{theorem}
\label{thm_reduced_system}
The numerical solution of the HDG method  (\ref{HDG_formulation}) satisfies 
\begin{equation*}
\underline{\boldsymbol{\sigma}}_h =  \underline{\boldsymbol{Q}}\widehat{\boldsymbol{u}}_h
+\underline{\boldsymbol{Q}}_{S}\boldsymbol{f}, \quad 
\boldsymbol{u}_h = \boldsymbol{U}\widehat{\boldsymbol{u}}_h+\boldsymbol{U}_{S}\boldsymbol{f}.
\end{equation*}
If we assume the boundary data $\boldsymbol{g}=0$, 
then $\widehat{\boldsymbol{u}}_h\in \boldsymbol{M}^0_{h}$ is the solution of 
\begin{equation}
\label{reduced_system}
 a_{h} (\widehat{\boldsymbol{u}}_h, \boldsymbol{\mu}) = 
\langle  (\underline{\boldsymbol{Q}}_{S}\boldsymbol{f}) \n
 - \tau \Pim  \boldsymbol{U}_{S}\boldsymbol{f}, \boldsymbol{\mu}\rangle_{\partial \mathcal{T}_{h}},\quad 
 \forall \boldsymbol{\mu}\in \boldsymbol{M}^0_{h},
\end{equation}
where
\[ 
a_{h}(\widehat{\boldsymbol{u}}_h, \boldsymbol{\mu}) 
= (\mathcal{A}\underline{\boldsymbol{Q}}\widehat{\boldsymbol{u}}_h,  
\underline{\boldsymbol{Q}}\boldsymbol{\mu})_{\mathcal{T}_{h}}
+ \langle \tau(\Pim  \boldsymbol{U}\widehat{\boldsymbol{u}}_h - \widehat{\boldsymbol{u}}_h ), 
\Pim  \boldsymbol{U}\boldsymbol{\mu} - \boldsymbol{\mu} \rangle_{\partial \mathcal{T}_{h}}.
\]
In addition, the bilinear operator $a_{h}(\boldsymbol{\lambda}, \boldsymbol{\lambda})$ is positive definite.
\end{theorem}

\begin{proof}
In order to show (\ref{reduced_system}) is true, we only need to show that for all 
$\boldsymbol{\lambda}, \boldsymbol{\mu}\in \boldsymbol{M}^0_h$, then 
\begin{equation*}
a_{h}(\boldsymbol{\lambda}, \boldsymbol{\mu}) = 
\langle  (\underline{\boldsymbol{Q}}\boldsymbol{\lambda}) \n
 - \tau (\Pim  \boldsymbol{U}\boldsymbol{\lambda} 
 - \boldsymbol{\lambda}), \boldsymbol{\mu}\rangle_{\partial \mathcal{T}_{h}}.
\end{equation*}

According to (\ref{local_solvers}), we have 
\begin{subequations}
\label{local_solvers_derivation}
\begin{align}
\label{local_solvers_derivation_eq1}
& (\mathcal{A} \underline{\boldsymbol{Q}}\boldsymbol{m},  \underline{\boldsymbol{v}})_{\mathcal{T}_{h}}
+(\boldsymbol{U}\boldsymbol{m}, \nabla\cdot \underline{\boldsymbol{v}})_{\mathcal{T}_{h}} 
= \langle \boldsymbol{m}, \underline{\boldsymbol{v}}\cdot \boldsymbol{n}\rangle_{\partial \mathcal{T}_{h}},\\
\label{local_solvers_derivation_eq2}
& (\nabla\cdot \underline{\boldsymbol{Q}}\boldsymbol{m}, \boldsymbol{\omega})_{\mathcal{T}_{h}} 
= \langle\tau(\Pim \boldsymbol{U}\boldsymbol{m}-\boldsymbol{m}),\boldsymbol{\omega}\rangle_{\partial \mathcal{T}_{h}},
\end{align}
\end{subequations}
for all $(\underline{\boldsymbol{v}}, \boldsymbol{\omega})\in \underline{\boldsymbol{V}}_{h}\times \boldsymbol{W}_{h}$, 
$\boldsymbol{m}\in \boldsymbol{M}^0_h$. Then, we have 
\begin{align*}
& \langle  (\underline{\boldsymbol{Q}}\boldsymbol{\lambda}) \n
 - \tau (\Pim  \boldsymbol{U}\boldsymbol{\lambda} 
 - \boldsymbol{\lambda}), \boldsymbol{\mu}\rangle_{\partial \mathcal{T}_{h}}\\
= & \langle \boldsymbol{\mu},  (\underline{\boldsymbol{Q}}\boldsymbol{\lambda}) \n\rangle_{\partial\mathcal{T}_{h}} 
- \langle  \tau (\Pim  \boldsymbol{U}\boldsymbol{\lambda} 
 - \boldsymbol{\lambda}), \boldsymbol{\mu}\rangle_{\partial \mathcal{T}_{h}}\\
= & (\mathcal{A} \underline{\boldsymbol{Q}}\boldsymbol{\mu}, 
\underline{\boldsymbol{Q}}\boldsymbol{\lambda})_{\mathcal{T}_{h}}
+(\boldsymbol{U}\boldsymbol{\mu}, \nabla\cdot \underline{\boldsymbol{Q}}\boldsymbol{\lambda})_{\mathcal{T}_{h}}  
- \langle  \tau (\Pim  \boldsymbol{U}\boldsymbol{\lambda} 
 - \boldsymbol{\lambda}), \boldsymbol{\mu}\rangle_{\partial \mathcal{T}_{h}}\quad
 \text{ by }(\ref{local_solvers_derivation_eq1})\\
 = & (\mathcal{A} \underline{\boldsymbol{Q}}\boldsymbol{\mu}, 
\underline{\boldsymbol{Q}}\boldsymbol{\lambda})_{\mathcal{T}_{h}}
+(\nabla\cdot \underline{\boldsymbol{Q}}\boldsymbol{\lambda}, \boldsymbol{U}\boldsymbol{\mu})_{\mathcal{T}_{h}}  
- \langle  \tau (\Pim  \boldsymbol{U}\boldsymbol{\lambda} 
 - \boldsymbol{\lambda}), \boldsymbol{\mu}\rangle_{\partial \mathcal{T}_{h}}\\
 = & (\mathcal{A} \underline{\boldsymbol{Q}}\boldsymbol{\mu}, 
\underline{\boldsymbol{Q}}\boldsymbol{\lambda})_{\mathcal{T}_{h}}
+ \langle  \tau (\Pim  \boldsymbol{U}\boldsymbol{\lambda} -\boldsymbol{\lambda}), \boldsymbol{U}\boldsymbol{\mu}-\boldsymbol{\mu}\rangle_{\partial \mathcal{T}_{h}}\quad
\text{ by } (\ref{local_solvers_derivation_eq2})\\
= & (\mathcal{A} \underline{\boldsymbol{Q}}\boldsymbol{\mu}, 
\underline{\boldsymbol{Q}}\boldsymbol{\lambda})_{\mathcal{T}_{h}}
+ \langle  \tau (\Pim  \boldsymbol{U}\boldsymbol{\lambda} -\boldsymbol{\lambda}), \Pim\boldsymbol{U}\boldsymbol{\mu}-\boldsymbol{\mu}\rangle_{\partial \mathcal{T}_{h}}\\
= & a_{h}(\boldsymbol{\lambda}, \boldsymbol{\mu}).
\end{align*}
So, we can conclude that (\ref{reduced_system}) holds. We end the proof by showing the bilinear operator $a_h(\cdot, \cdot)$ is positive definite.

If $a_{h}(\boldsymbol{\lambda}, \boldsymbol{\lambda}) = 0$ for 
some $\boldsymbol{\lambda}\in \boldsymbol{M}^0_h$, from the previous result we have 
\[
\underline{\boldsymbol{Q}}\boldsymbol{\lambda} = 0, \quad \Pim\boldsymbol{U}\boldsymbol{\lambda} - \boldsymbol{\lambda}|_{\partial\mathcal{T}_{h}}=0. 
\]
We apply integration by parts on (\ref{local_solvers_eq1}), we have 
\[
\bintK{\underline{\epsilon} (\boldsymbol{U}\boldsymbol{\lambda})}{\underline{\boldsymbol{v}}} = 0, \quad \forall \, \underline{\boldsymbol{v}} \in \underline{\boldsymbol{V}}(K).
\]
This implies that $\underline{\epsilon}(\boldsymbol{U}\boldsymbol{\lambda})|_{K} = 0$ for all $K\in\mathcal{T}_{h}$. 
So, for any $K\in\mathcal{T}_{h}$, there are $\boldsymbol{a}_{K}, \boldsymbol{b}_{K}\in\mathbb{R}^{3}$ 
such that $\boldsymbol{U}\boldsymbol{\lambda}|_{K} = \boldsymbol{a}_{K}\times\boldsymbol{x}+\boldsymbol{b}_{K}$. 
Since $k \ge 1$, we have $\boldsymbol{P}_M \boldsymbol{U} \boldsymbol{\lambda} = \boldsymbol{U} \boldsymbol{\lambda}$. 
Combining this result with the fact that $\Pim\boldsymbol{U}\boldsymbol{\lambda} - \boldsymbol{\lambda}|_{\partial\mathcal{T}_{h}}=0$ and 
$\boldsymbol{\lambda}|_{\partial \Omega}=0$, we can conclude that $\boldsymbol{U}\boldsymbol{\lambda} \in \boldsymbol{C}^0(\Omega)$ and $\boldsymbol{U}\boldsymbol{\lambda}|_{\partial \Omega} = 0$. 

Finally, let us consider two adjacent element $K_1, K_2$ with the interface $F = \bar{K_1} \cap \bar{K_2}$. In addition, we assume that on $K_i$, $\boldsymbol{U}\boldsymbol{\lambda}$ can be expressed as
\[
\boldsymbol{U}\boldsymbol{\lambda} = \boldsymbol{a}_i \times \boldsymbol{x} + \boldsymbol{b}_i, \quad i = 1, 2.
\]
We claim that $\boldsymbol{a}_1 = \boldsymbol{a}_2$ and $\boldsymbol{b}_1 = \boldsymbol{b}_2$. This fact can be shown by considering the 
continuity of the function on the interface $F$. We omit the detailed proof since it only involves elementary linear algebra. 

From this result we conclude that there exist $\boldsymbol{a}, \boldsymbol{b}\in\mathbb{R}^{3}$ such that 
$\boldsymbol{U \lambda} = \boldsymbol{a} \times \boldsymbol{x} + \boldsymbol{b}$ in $\Omega$. 
By the fact that $\boldsymbol{U \lambda}|_{\partial \Omega} = 0$, we can conclude that $\boldsymbol{U \lambda} = 0$, 
hence $\boldsymbol{\lambda} = 0$. This completes the proof.
\end{proof}

\begin{remark}
In Theorem~\ref{thm_reduced_system}, we assume the boundary data $\boldsymbol{g}=0$. 
Actually, if $\boldsymbol{g}$ is not zero, we can still obtain the same linear system as $a_{h}$ in Theorem~\ref{thm_reduced_system}
by the same treatment of boundary data in \cite{CockburnDuboisGopalakrishnanTan2013}.
\end{remark}

\section{Error Analysis}
In this section we provide detailed proofs for our a priori error estimates - Theorem~\ref{Main_result} 
and Theorem~\ref{thm_locking_free}. We use elementwise Korn's inequality (Lemma~\ref{symmetric_grad}), 
which is novel and crucial in error analysis. We use $\Piv, \Piw$ to denote the standard $L^2$-orthogonal 
projection onto $\underline{\boldsymbol{V}}_h, \boldsymbol{W}_h$, respectively. In addition, we denote
\[
\esigma = \Piv \underline{\boldsymbol{\sigma}} - \underline{\boldsymbol{\sigma}}_h, \quad 
\eu = \Piw \boldsymbol{u} - \boldsymbol{u}_h, \quad
\euhat  = \Pim \boldsymbol{u} - \widehat{\boldsymbol{u}}_h,
\]

In the analysis, we are going to use the following classical results:
\begin{subequations}\label{classical_ineq}
\begin{alignat}{2}
\label{classical_ineq_1}
\|\boldsymbol{u} - \Piw \boldsymbol{u}\|_{\Omega} & \le C h^s \|\boldsymbol{}u\|_{s, \Omega} \quad && 0 \le s \le k+2, \\
\label{classical_ineq_2}
\|\underline{\boldsymbol{\sigma}} - \Piv \underline{\boldsymbol{\sigma}}\|_{\Omega} & \le C h^t \|\underline{\boldsymbol{\sigma}}\|_{t, \Omega} && 0 \le t \le k+1, \\
\label{classical_ineq_3}
\|\boldsymbol{u} - \Pim \boldsymbol{u}\|_{\Eh} &\le C h^{s-\frac12} \|\boldsymbol{u}\|_{s, \Omega},  \quad && 1\le s \le k+1,\\
\label{classical_ineq_4}
\|\boldsymbol{u} - \Piw \boldsymbol{u}\|_{\partial K} &\le C h^{s-\frac12} \|\boldsymbol{u}\|_{s, K}, \quad && 1 \le s \le k+2, \\
\label{classical_ineq_5}
\|\underline{\boldsymbol{\sigma}}\n - \Piv \underline{\boldsymbol{\sigma}}\n\|_{\partial K} &\le C h^{t-\frac12} \|\underline{\boldsymbol{\sigma}}\|_{t, K}, \quad && 1 \le t \le k+1, \\
\label{classical_ineq_6}
\|\boldsymbol{v}\|_{\partial K} &\le C h^{-\frac12} \|\boldsymbol{v}\|_K, \quad &&\forall \; \boldsymbol{v} \in \boldsymbol{P}_s(K), \\
\label{classical_ineq_7}
\|\underline{\boldsymbol{\sigma}}\n - \Pim (\underline{\boldsymbol{\sigma}}\n)\|_{\partial K} &\le C h^{t-\frac12} \|\underline{\boldsymbol{\sigma}}\|_{t, K}, \quad && 1 \le t \le k+1.
\end{alignat}
\end{subequations}
The above results are due to standard approximation theory of polynomials, trace inequality.

Let $\underline{\boldsymbol{\epsilon}}_{h}$ denote the discrete symmetric gradient operator, 
such that for any $K\in\Oh$, $\underline{\boldsymbol{\epsilon}}_{h}|_{K}= \underline{\boldsymbol{\epsilon}}|_{K}$. 
It is well known (see Theorem $2.2$ in \cite{Ciarlet2010}) the {\em{kernel}} of the  operator $\underline{\boldsymbol{\epsilon}}_{h} (\cdot)$ is:
\[
\ker \underline{\boldsymbol{\epsilon}}_{h} = \Upsilon_h := \{\boldsymbol{\Lambda} \in \boldsymbol{L}^2(\Omega), \; \boldsymbol{\Lambda}|_{K} = \underline{\boldsymbol{B}}_{K} \boldsymbol{x} + \boldsymbol{b}_{K}, \; \underline{\boldsymbol{B}}_{K}\in \underline{\boldsymbol{\mathcal{A}}}, \boldsymbol{b}_{K} \in \mathbb{R}^3,K\in\Oh \}.
\]
Here, $\underline{\boldsymbol{\mathcal{A}}}$ denotes the set of all anti-symmetric matrices in $\mathbb{R}^{3\times 3}$.

In the analysis, we need the following elementwise Korn's inequality:
\begin{lemma}\label{symmetric_grad}
Let $K \in \Oh$ be a generic element with size $h_K$ and $\Upsilon(K):= \Upsilon_h|_K$. Then for any function 
$\boldsymbol{v} \in \boldsymbol{W}(K)$, we have
\[
\inf_{\boldsymbol{\Lambda} \in \Upsilon(K)} \|\nabla(\boldsymbol{v} + \boldsymbol{\Lambda})\|_K 
\le C \|\underline{\boldsymbol{\epsilon}}(\boldsymbol{v})\|_K,
\]
Here $C$ is independent of the size $h_K$. In addition, if $K$ is a tetrahedron, 
the above inequality holds for any $\boldsymbol{v}\in \boldsymbol{H}^1(K)$.
\end{lemma}
\begin{proof}
Let $\widehat{K}$ denote the reference tetrahedron element and $\boldsymbol{v} \in \boldsymbol{H}^1(K)$.
The mapping from $\widehat{K}$ to $K$ is $\boldsymbol{x} = 
\underline{\boldsymbol{A}}_K \widehat{\boldsymbol{x}}+\boldsymbol{c}_{K}$ 
where $\underline{\boldsymbol{A}}_K$ is a non-singular matrix and  $\boldsymbol{c}_{K}\in\mathbb{R}^{3}$.

We define $\widehat{\boldsymbol{v}}$, which is the pull back of $\boldsymbol{v}$ on $\widehat{K}$, by
\begin{align*}
\underline{\boldsymbol{A}}_K^{-\top} \widehat{\boldsymbol{v}}(\widehat{\boldsymbol{x}}) 
= \boldsymbol{v}(\boldsymbol{x}) \quad \forall \widehat{\boldsymbol{x}} \in \widehat{K}.
\end{align*}
So, we have 
\begin{align*}
\nabla \boldsymbol{v} (\boldsymbol{x}) = \nabla (\underline{\boldsymbol{A}}_K^{-\top} 
\widehat{\boldsymbol{v}})(\boldsymbol{x}) = \underline{\boldsymbol{A}}_K^{-\top} 
(\nabla\widehat{\boldsymbol{v}})(\boldsymbol{x}).
\end{align*}
The last equality above is due to the fact that every component of 
$\underline{\boldsymbol{A}}_K^{-\top} $ is constant. It is easy to see that 
\begin{align*}
(\nabla\widehat{\boldsymbol{v}})(\boldsymbol{x}) = \widehat{\nabla} 
\widehat{\boldsymbol{v}}(\widehat{\boldsymbol{x}})\underline{\boldsymbol{A}}_K^{-1}.
\end{align*}
So, we have 
\begin{align*}
\underline{\boldsymbol{A}}_K^{-\top}\widehat{\nabla} \widehat{\boldsymbol{v}}(\widehat{\boldsymbol{x}}) 
\underline{\boldsymbol{A}}_K^{-1} = \nabla  \boldsymbol{v}(\boldsymbol{x}).
\end{align*}
By taking the symmetric part of both sides of the above equation, we have 
\begin{align}
\label{tensor_transform}
\underline{\boldsymbol{A}}_K^{-\top}\widehat{\underline{\boldsymbol{\epsilon}}} 
\widehat{\boldsymbol{v}}(\widehat{\boldsymbol{x}}) 
\underline{\boldsymbol{A}}_K^{-1} = \underline{\boldsymbol{\epsilon}}  \boldsymbol{v}(\boldsymbol{x}).
\end{align}

According to Theorem $2.3$ in \cite{Ciarlet2010}, the following inequality holds: 
\begin{align*}
\inf_{\widehat{\boldsymbol{\Lambda}} \in \Upsilon(\widehat{K})} \|\widehat{\boldsymbol{v}} 
+\widehat{\boldsymbol{\Lambda}}\|_{1, \widehat{K}} \le C \|\widehat{\underline{\boldsymbol{\epsilon}}}
(\widehat{\boldsymbol{v}})\|_{0, \widehat{K}}.
\end{align*}
So, there is $\widehat{\boldsymbol{\Lambda}} = \underline{\boldsymbol{B}}_{\widehat{K}} \widehat{\boldsymbol{x}}
+ \boldsymbol{b}_{\widehat{K}}$ with $\underline{\boldsymbol{B}}_{\widehat{K}}\in \underline{\boldsymbol{\mathcal{A}}}$ 
and $\boldsymbol{b}_{\widehat{K}} \in \mathbb{R}^3$, such that  
\begin{align}
\label{reference_ele}
\|\widehat{\nabla} (\widehat{\boldsymbol{v}} +\widehat{\boldsymbol{\Lambda}}) \|_{0, \widehat{K}} 
\le C \|\widehat{\underline{\boldsymbol{\epsilon}}}(\widehat{\boldsymbol{v}})\|_{0, \widehat{K}}.
\end{align}

We define 
\begin{align*}
\boldsymbol{\Lambda} (\boldsymbol{x}) = \underline{\boldsymbol{A}}_K^{-\top}
\widehat{\boldsymbol{\Lambda}}(\widehat{\boldsymbol{x}})\quad \forall \boldsymbol{x} \in K.
\end{align*}
It is easy to see that
\begin{align*}
\nabla \boldsymbol{\Lambda} = \underline{\boldsymbol{A}}_K^{-\top}\widehat{\nabla} 
\widehat{\boldsymbol{\Lambda}} \underline{\boldsymbol{A}}_K^{-1} 
= \underline{\boldsymbol{A}}_K^{-\top}\underline{\boldsymbol{B}}_{\widehat{K}}
\underline{\boldsymbol{A}}_K^{-1}\in \underline{\boldsymbol{\mathcal{A}}}.
\end{align*}
So, $\boldsymbol{\Lambda} \in \Upsilon(K)$.
Then, by standard scaling argument with (\ref{tensor_transform}, \ref{reference_ele}) and the shape regularity of the meshes, 
we can conclude that the proof for arbitrary tetrahedron element is complete.

Now, we consider the case of arbitrary shape regular element $K$, which can be hexahedron, prism or pyramid.
Let $\boldsymbol{v}= (v_{1}, v_{2}, v_{3})^{\top}\in \boldsymbol{W}_{h}|_{K}$. 
It is well known that for any $1\leq i,j,k\leq 3$, 
\begin{align*}
\partial_{j} (\partial_{k} v_{i}) = \partial_{j} (\epsilon_{ik}(\boldsymbol{v})) 
+ \partial_{k} (\epsilon_{ij}(\boldsymbol{v})) - \partial_{i} (\epsilon_{jk}(\boldsymbol{v})). 
\end{align*}
Here, $\epsilon_{ik}(\boldsymbol{v}) = \left(\underline{\boldsymbol{\epsilon}}(\boldsymbol{v})\right)_{ik}$.
Consequently, we have
\begin{align*}
\Vert \nabla (\partial_{j} v_{i} - \partial_{i} v_{j}) \Vert_{0,K} 
\leq C \Vert \nabla \underline{\boldsymbol{\epsilon}} (\boldsymbol{v}) \Vert_{0,K}
\leq C h_{K}^{-1}\Vert \underline{\boldsymbol{\epsilon}} (\boldsymbol{v}) \Vert_{0,K}.
\end{align*}
We define an anti-symmetric matrix $\underline{\boldsymbol{B}}_{K}$ by
\begin{align*}
\left( \underline{\boldsymbol{B}}_{K} \right)_{ij} = \dfrac{1}{2\vert K\vert}\int_{K} (\partial_{j} v_{i} - \partial_{i} v_{j}) d\boldsymbol{x}
\quad   1\leq i,j\leq 3.
\end{align*}
We take $\boldsymbol{\Lambda} = \underline{\boldsymbol{B}}_{K}\boldsymbol{x}$, 
which is obviously in $\Upsilon({K})$. Then, we have 
\begin{align*}
\int_{K} \left( \nabla (\boldsymbol{v}-\boldsymbol{\Lambda})
 - \underline{\boldsymbol{\epsilon}}(\boldsymbol{v}) \right) d\boldsymbol{x}= \int_K (\nabla \boldsymbol{v} - \underline{\boldsymbol{\epsilon}} (\boldsymbol{v}) ) \, d \boldsymbol{x} - \underline{\boldsymbol{B}}_K \int_K 1 \, d \boldsymbol{x} = 0.
\end{align*}
By the Poincar\'{e} inequality, we have 
\begin{align*}
\Vert \nabla (\boldsymbol{v}-\boldsymbol{\Lambda}) 
- \underline{\boldsymbol{\epsilon}}(\boldsymbol{v})\Vert_{0, K} 
\leq  C h_{K} \sum_{1\leq i,j\leq 3}\Vert \nabla (\partial_{j} v_{i} - \partial_{i} v_{j}) \Vert_{0,K} 
\leq C \Vert \underline{\boldsymbol{\epsilon}} (\boldsymbol{v}) \Vert_{0,K}.
\end{align*}
We immediately have that 
\begin{align*}
\Vert \nabla (\boldsymbol{v} - \boldsymbol{\Lambda})\Vert_{0, K} 
\leq C \Vert \underline{\boldsymbol{\epsilon}}(\boldsymbol{v})\Vert_{0, K}. 
\end{align*}
This completes the proof.
\end{proof}

{\textbf{Step 1: The error equation.}} We first present the error equation for the analysis.
\begin{lemma}
Let $(\boldsymbol{u}, \underline{\boldsymbol{\sigma}}), (\boldsymbol{u}_h, \underline{\boldsymbol{\sigma}}_h, \widehat{\boldsymbol{u}}_h)$ solve \eqref{elasticity_equation} and \eqref{HDG_formulation} respectively, we have
\begin{subequations}\label{error_equation}
\begin{alignat}{1}
\label{error_equation_1}
\bint{\mathcal{A} \esigma}{\underline{\boldsymbol{v}}} + \bint{\eu}{\nabla \cdot \underline{\boldsymbol{v}}}
- \bintEh{\euhat}{\underline{\boldsymbol{v}} \n} &=  \bint{\mathcal{A} (\Piv \underline{\boldsymbol{\sigma}} - \underline{\boldsymbol{\sigma}}) }{\underline{\boldsymbol{v}}},\\
\label{error_equation_2}
\bint{\esigma}{\nabla \boldsymbol{\omega}} - \bintEh{\underline{\boldsymbol{\sigma}} \n - \widehat{\underline{\boldsymbol{\sigma}}}_h \n}{\boldsymbol{\omega}} & = 0,\\
\label{error_equation_3}
\bintEhi{\underline{\boldsymbol{\sigma}} \n - \widehat{\underline{\boldsymbol{\sigma}}}_h \n}{\boldsymbol{\mu}} & = 0,\\
\label{error_equation_4}
\bintEhb{\euhat}{\boldsymbol{\mu}} & = 0,
\end{alignat} 
\end{subequations}
for all $(\underline{\boldsymbol{v}}, \boldsymbol{\omega}, \boldsymbol{\mu}) \in \underline{\boldsymbol{V}}_h \times \boldsymbol{W}_h \times \boldsymbol{M}_h$.
\end{lemma}

\begin{proof}
We notice that the exact solution $(\boldsymbol{u}, \underline{\boldsymbol{\sigma}}, \boldsymbol{u}|_{\Eh})$ also satisfies the equation \eqref{HDG_formulation}.
Hence, {after simple algebraic manipulations, we get that}
\begin{alignat*}{1}
\bint{\mathcal{A} \Piv \underline{\boldsymbol{\sigma}}}{\underline{\boldsymbol{v}}} + \bint{\Piw \boldsymbol{u}}{\nabla \cdot \underline{\boldsymbol{v}}}
- \bintEh{\Pim \boldsymbol{u}}{\underline{\boldsymbol{v}} \n} &=  \\
-\bint{\mathcal{A}(\deltasigma)}{\underline{\boldsymbol{v}}} +  \bintEh{\boldsymbol{u} - \Pim \boldsymbol{u}}{\underline{\boldsymbol{v}}\n} &- \bint{\deltau}{\nabla \cdot \underline{\boldsymbol{v}}} ,\\
\bint{\Piv \underline{\boldsymbol{\sigma}}}{\nabla \boldsymbol{\omega}} -
\bintEh{\underline{\boldsymbol{\sigma}} \n}{\boldsymbol{\omega}}  
= - \bint{\boldsymbol{f}}{\boldsymbol{\omega}} - &\bint{\deltasigma}{\nabla
  \boldsymbol{\omega}}
\\
\bintEhi{\underline{\boldsymbol{\sigma}} \n}{\boldsymbol{\mu}} & = 0,\\
\bintEhb{\Pim \boldsymbol{u}}{\boldsymbol{\mu}} & = - \bintEhb{\boldsymbol{u} - \Pim \boldsymbol{u}}{\boldsymbol{\mu}},
\end{alignat*} 
{for all $(\underline{\boldsymbol{v}}, \boldsymbol{w}, \boldsymbol{\mu}) \in \underline{\boldsymbol{V}}_h \times \boldsymbol{W}_h \times \boldsymbol{M}_h$.}
Notice that the local spaces satisfy the following inclusion property:
\[
\nabla \cdot \underline{\boldsymbol{V}}(K) \subset \boldsymbol{W}(K), \quad \underline{\boldsymbol{\epsilon}} (\boldsymbol{W}(K)) \subset \underline{\boldsymbol{V}}(K), \quad \underline{\boldsymbol{V}}(K) \n |_{F} \subset \boldsymbol{M}(F).
\]
Hence by the property of the $L^2$-projection, the above system can be simplified as:
\begin{alignat*}{1}
\bint{\mathcal{A} \Piv \underline{\boldsymbol{\sigma}}}{\underline{\boldsymbol{v}}} + \bint{\Piw \boldsymbol{u}}{\nabla \cdot \underline{\boldsymbol{v}}}
- \bintEh{\Pim \boldsymbol{u}}{\underline{\boldsymbol{v}} \n} &= -\bint{\mathcal{A}(\deltasigma)}{\underline{\boldsymbol{v}}}, \\
\bint{\Piv \underline{\boldsymbol{\sigma}}}{\nabla \boldsymbol{\omega}} -
\bintEh{\underline{\boldsymbol{\sigma}} \n}{\boldsymbol{\omega}}  
&= - \bint{\boldsymbol{f}}{\boldsymbol{\omega}}, \\
\bintEhi{\underline{\boldsymbol{\sigma}} \n}{\boldsymbol{\mu}} & = 0,\\
\bintEhb{\Pim \boldsymbol{u}}{\boldsymbol{\mu}} & = 0,
\end{alignat*} 
for all $(\underline{\boldsymbol{v}}, \boldsymbol{w}, \boldsymbol{\mu}) \in \underline{\boldsymbol{V}}_h \times \boldsymbol{W}_h \times \boldsymbol{M}_h$. Here we applied the fact that $\bint{\deltasigma}{\nabla
  \boldsymbol{\omega}} = \bint{\deltasigma}{{\underline{\boldsymbol{\epsilon}} (\boldsymbol{\omega}}) }
 = 0$. If we now subtract the equations \eqref{HDG_formulation}, we obtain the result. This completes the proof.
\end{proof}

{\textbf{Step 2: Estimate of $\esigma$.}} We are now ready to obtain our first estimate. 
\begin{proposition}\label{energy_argument}
We have
\[
\bint{\mathcal{A} \esigma}{\esigma}  + \bintEh{\tau (\Pim \eu - \euhat)}{\Pim \eu - \euhat} = -\bint{\mathcal{A}(\deltasigma)}{\esigma}+ T_1 - T_2,
\]
where $T_1, T_2$ are defined as:
\begin{align*}
T_1 & := \bintEh{\eu - \euhat}{\underline{\boldsymbol{\sigma}} \n - (\Piv \underline{\boldsymbol{\sigma}}) \n}, \\
T_2 & := \bintEh{\eu - \euhat}{\tau (\Pim (\deltau))}.
\end{align*}
\end{proposition}

\begin{proof}
By the error equation \eqref{error_equation_4} we know that $\euhat = 0$ on $\partial \Omega$. This implies that
\[
\bintEhb{\euhat}{\underline{\boldsymbol{\sigma}} \n - \widehat{\underline{\boldsymbol{\sigma}}}_h \n} = 0.
\]
Now taking $(\underline{\boldsymbol{v}}, \boldsymbol{w}, \boldsymbol{\mu}) = (\esigma, \eu, \euhat)$ in error equations \eqref{error_equation_1} - \eqref{error_equation_3} and adding these equations together with the above identity, we obtain, after some algebraic manipulation, 
\begin{equation}\label{energy_1}
\bint{\mathcal{A} \esigma}{\esigma} + \bintEh{\eu - \euhat}{\esigma \n - (\underline{\boldsymbol{\sigma}} \n - \widehat{\underline{\boldsymbol{\sigma}}}_h \n)} = - \bint{\mathcal{A}(\deltasigma)}{\esigma}.
\end{equation}
Now we work with the second term on the left hand side, 
\begin{align*}
 \esigma \n - (\underline{\boldsymbol{\sigma}} \n - \widehat{\underline{\boldsymbol{\sigma}}}_h \n) &= \Piv \underline{\boldsymbol{\sigma}} \n - \underline{\boldsymbol{\sigma}}_h \n - \underline{\boldsymbol{\sigma}} \n + \widehat{\underline{\boldsymbol{\sigma}}}_h \n 
\intertext{by the definition of the numerical trace \eqref{numerical_trace}, }
& = -(\deltasigma) \n - \tau(\Pim \boldsymbol{u}_h - \widehat{\boldsymbol{u}}_h), \\
& = -(\deltasigma) \n + \tau (\Pim \eu - \euhat) - \tau(\Pim( \Piw \boldsymbol{u} - \boldsymbol{u})), 
\end{align*}
the last step is by the definition of $\eu, \euhat$. Inserting the above identity into \eqref{energy_1}, moving terms around, we have
\[
\bint{\mathcal{A} \esigma}{\esigma}  + \bintEh{\eu - \euhat}{\tau(\Pim \eu - \euhat)} = -\bint{\mathcal{A}(\deltasigma)}{\esigma}+ T_1 - T_2.
\]
Finally, notice that on each $F \in \partial \Oh$, $\tau(\Pim \eu - \euhat)|_F \in \boldsymbol{M}(F)$, so we have
\[
\bintEh{\eu - \euhat}{\tau(\Pim \eu - \euhat)} = \bintEh{\Pim \eu - \euhat}{\tau(\Pim \eu - \euhat)}.
\]
This completes the proof.
\end{proof}

From the above energy argument we can see that we need to bound $T_1, T_2$ in order to have an estimate for $\esigma$. Next we present the estimates for these two terms:
\begin{lemma}\label{Two_terms}
If the parameter $\tau = \mathcal{O}(h^{-1})$, we have
\begin{align*}
T_1 & \le C h^{t} \|\underline{\boldsymbol{\sigma}}\|_{t, \Omega} \, (\|\tau^{\frac12}(\Pim \eu - \euhat)\|_{\partial \Oh} + \|\underline{\boldsymbol{\epsilon}} (\eu)\|_{\Oh}) \\
T_2 & \le C h^{s-1} \|\boldsymbol{u}\|_{s, \Omega} \, \|\tau^{\frac12}(\Pim \eu - \euhat)\|_{\partial \Oh} ,
\end{align*}
for all $1 \le t \le k+1, 1 \le s \le k+2$.
\end{lemma}
\begin{proof}
We first bound $T_2$. We have
\begin{align*}
T_2 &=  \bintEh{\eu - \euhat}{\tau (\Pim (\deltau))} =  \bintEh{\Pim \eu - \euhat}{\tau (\Pim (\deltau))} \\
& =  \bintEh{\Pim \eu - \euhat}{\tau (\deltau)} \\
&\le \|\tau^{\frac12}(\Pim \eu - \euhat)\|_{\partial \Oh} \tau^{\frac12} \|\boldsymbol{u} - \Piw \boldsymbol{u}\|_{\partial \Oh}\\
& \le C h^{s} (\tau^{\frac12} h^{-\frac12}) \|\tau^{\frac12}(\Pim \eu - \euhat)\|_{\partial \Oh} \|\boldsymbol{u}\|_{s, \Omega}, 
\end{align*}
for all $1 \le s \le k+2$. The last step we applied the inequality \eqref{classical_ineq_4}. 

The estimate for $T_1$ is much more sophisticated. We first split $T_1$ into two parts:
\[T_1 = T_{11} + T_{12},\]
where
\begin{align*}
T_{11} & := \bintEh{\Pim \eu - \euhat}{\underline{\boldsymbol{\sigma}} \n - (\Piv \underline{\boldsymbol{\sigma}}) \n}, \\
T_{12} &:= \bintEh{\eu - \Pim \eu}{\underline{\boldsymbol{\sigma}} \n - (\Piv \underline{\boldsymbol{\sigma}}) \n}.
\end{align*}

For $T_{11}$, we simply apply the Cauchy-Schwarz inequality,
\begin{align*}
T_{11} &\le \|\tau^{\frac12}(\Pim \eu - \euhat)\|_{\partial \Oh} \, \tau^{-\frac12}\|\underline{\boldsymbol{\sigma}} \n - (\Piv \underline{\boldsymbol{\sigma}}) \n\|_{\partial \Oh}\\
&\le C h^{t} (\tau^{-\frac12}h^{-\frac12}) \|\underline{\boldsymbol{\sigma}}\|_{t, \Omega} \|\tau^{\frac12}(\Pim \eu - \euhat)\|_{\partial \Oh},
\end{align*}
for all $1 \le t \le k+1$. Here we used the inequality \eqref{classical_ineq_5}.

Now we work on $T_{12}$. Using the $L^2$-orthogonal property of the projection $\Pim$, we can write 
\begin{align*}
T_{12} &= \bintEh{\eu - \Pim \eu}{\underline{\boldsymbol{\sigma}} \n - (\Piv \underline{\boldsymbol{\sigma}}) \n} \\
&= \bintEh{\eu - \Pim \eu}{\underline{\boldsymbol{\sigma}} \n } \\
 & = \bintEh{\eu - \Pim \eu}{\underline{\boldsymbol{\sigma}} \n - \Pim (\underline{\boldsymbol{\sigma}} \n)}, 
\intertext{by the fact $\Piv \underline{\boldsymbol{\sigma}} \n|_{F}, \Pim (\underline{\boldsymbol{\sigma}} \n)|_{F} \in \boldsymbol{M}(F)$ for all $F \in \partial \Oh$,}
T_{12} & =\bintEh{\eu}{\underline{\boldsymbol{\sigma}} \n - \Pim (\underline{\boldsymbol{\sigma}} \n)}, \quad \text{since $\Pim \eu|_F \in \boldsymbol{M}(F), \forall F \in \partial \Oh$,} \\
& = \bintEh{\eu + \boldsymbol{\Lambda}}{\underline{\boldsymbol{\sigma}} \n - \Pim (\underline{\boldsymbol{\sigma}} \n)},
\end{align*}
where $\boldsymbol{\Lambda} \in \boldsymbol{L}^2(\Omega)$ is any vector-valued function in $\Upsilon_h$. 
Notice here the last step holds only if $\Upsilon_h|_F \in \boldsymbol{M}(F), \; \forall F \in \partial \Oh$. This is true if $k \ge 1$.  Next, on each $K \in \Oh$, if we denote $\overline{\boldsymbol{u}}$ to be the average of $\boldsymbol{u}$ over $K$, then we have
\begin{align*}
\bintK{\eu + \boldsymbol{\Lambda}}{\underline{\boldsymbol{\sigma}} \n - \Pim (\underline{\boldsymbol{\sigma}} \n)} &= \bintK{\eu + \boldsymbol{\Lambda} - \overline{(\eu + \boldsymbol{\Lambda})}}{\underline{\boldsymbol{\sigma}} \n - \Pim (\underline{\boldsymbol{\sigma}} \n)}\\
& \le \|\eu + \boldsymbol{\Lambda} - \overline{(\eu + \boldsymbol{\Lambda})}\|_{\partial K} \|\underline{\boldsymbol{\sigma}} \n - \Pim (\underline{\boldsymbol{\sigma}} \n)\|_{\partial K},
\intertext{by the standard inequalities \eqref{classical_ineq_6}, \eqref{classical_ineq_7},}
\bintK{\eu + \boldsymbol{\Lambda}}{\underline{\boldsymbol{\sigma}} \n - \Pim (\underline{\boldsymbol{\sigma}} \n)} &\le C h^{t-1} \|\underline{\boldsymbol{\sigma}}\|_{t, K} \|\eu + \boldsymbol{\Lambda} - \overline{(\eu + \boldsymbol{\Lambda})}\|_{K}\\
&\le C h^t \|\underline{\boldsymbol{\sigma}}\|_{t, K} \|\nabla (\eu + \boldsymbol{\Lambda})\|_K,
\intertext{for all $1 \le t \le k+1$. The last step is by the  Poincar\'{e} inequality. 
Notice that the constant $C$ in above inequality is independent of $\boldsymbol{\Lambda}\in \Upsilon_h$. 
Now applying the Lemma \ref{symmetric_grad}, yields, }
\bintK{\eu + \boldsymbol{\Lambda}}{\underline{\boldsymbol{\sigma}} \n - \Pim (\underline{\boldsymbol{\sigma}} \n)} & \le C h^t \|\underline{\boldsymbol{\sigma}}\|_{t, K} \|\underline{\boldsymbol{\epsilon}} (\eu)\|_K,
\end{align*}
Sum over all $K \in \Oh$, we have 
\[
T_{12} \le C h^t \|\underline{\boldsymbol{\sigma}}\|_{t, \Omega} \|\underline{\boldsymbol{\epsilon}} (\eu)\|_{\Oh},
\]
for all $1 \le t \le k+1$. We complete the proof by combining the estimates for $T_2, T_{11}, T_{12}$.
\end{proof}
Combining Lemma \ref{Two_terms} and Proposition \ref{energy_argument}, we obtain the following estimate.
\begin{corollary}\label{estimate_1}
If the parameter $\tau = \mathcal{O}(h^{-1})$, then we have
\begin{align*}
\|\esigma\|^2_{L^{2}(\mathcal{A},\Omega)} &+ \|\tau^{\frac12}(\Pim \eu - \euhat)\|^2_{\partial \Oh} \\
&\le C \left( h^{2t} \|\underline{\boldsymbol{\sigma}}\|^2_{t, \Omega} + h^{2(s-1)} \|\boldsymbol{u}\|^2_{s, \Omega} + h^t \|\underline{\boldsymbol{\sigma}}\|_{t, \Omega} \|\underline{\boldsymbol{\epsilon}} (\eu)\|_{\Oh}\right),
\end{align*}
for all $1 \le s \le k+2, 1 \le t \le k+1$, the constant $C$ is independent of $h$ and exact solution.
\end{corollary}

The proof is omitted. One can obtain the above result by the Cauchy-Schwarz inequality and weighted Young's inequality. Finally, we can finish the 
estimate for $\esigma$ by the following estimate for $\underline{\boldsymbol{\epsilon}} (\eu)$:
\begin{lemma}\label{estimate_gradu}
Under the same assumption as Theorem \ref{estimate_1}, we have
\[
\|\underline{\boldsymbol{\epsilon}} (\eu)\|_{\Oh} \le C \left(h^t \|\underline{\boldsymbol{\sigma}}\|_{t, \Omega} 
+  \|\esigma\|_{L^{2}(\mathcal{A},\Omega)} +  \|\tau^{\frac12}(\Pim \eu - \euhat)\|_{\partial \Oh}\right),
\]
for all $0 \le t \le k+1$.
\end{lemma}
\begin{proof}
Notice that $\underline{\boldsymbol{\epsilon}} (\eu) \in \underline{\boldsymbol{V}}_h$, so we can take $\underline{\boldsymbol{v}}=\underline{\boldsymbol{\epsilon}} (\eu)$ in the error equation \eqref{error_equation_1}, after integrating by parts, we have:
\[
\bint{\mathcal{A} \esigma}{\underline{\boldsymbol{\epsilon}} (\eu)} - \bint{ \nabla \eu}{ \underline{\boldsymbol{\epsilon}} (\eu)}
+ \bintEh{\eu - \euhat}{\underline{\boldsymbol{\epsilon}} (\eu) \n} =  \bint{\mathcal{A} (\Piv \underline{\boldsymbol{\sigma}} - \underline{\boldsymbol{\sigma}}) }{\underline{\boldsymbol{\epsilon}} (\eu)}.
\]
Notice that $\underline{\boldsymbol{\epsilon}} (\eu) \in \underline{\boldsymbol{V}}_h$ and it is symmetric, we have 
\[
\bint{ \nabla \eu}{ \underline{\boldsymbol{\epsilon}} (\eu)} = \| \underline{\boldsymbol{\epsilon}} (\eu)\|^2_{\Oh}, \quad 
\bintEh{\eu - \euhat}{\underline{\boldsymbol{\epsilon}} (\eu) \n} = \bintEh{\Pim \eu - \euhat}{\underline{\boldsymbol{\epsilon}} (\eu) \n}.
\]
Inserting these two identities into the first equation, we have
\begin{align*}
\| \underline{\boldsymbol{\epsilon}} (\eu)\|^2_{\Oh} &= \bint{\mathcal{A} \esigma}{\underline{\boldsymbol{\epsilon}} (\eu)} + \bintEh{\Pim \eu - \euhat}{\underline{\boldsymbol{\epsilon}} (\eu) \n}\\
& \quad + \bint{\mathcal{A} (\underline{\boldsymbol{\sigma}} - \Piv \underline{\boldsymbol{\sigma}}) }{\underline{\boldsymbol{\epsilon}} (\eu)}\\
& \le C \|\esigma\|_{L^{2}(\mathcal{A},\Omega)} \| \underline{\boldsymbol{\epsilon}} (\eu)\|_{\Oh} + C \tau^{-\frac12}  \|\tau^{\frac12}(\Pim \eu - \euhat)\|_{\partial \Oh} \|\underline{\boldsymbol{\epsilon}} (\eu)\n\|_{\partial \Oh} \\
& \quad + C h^t \|\underline{\boldsymbol{\sigma}}\|_{t, \Omega}\|\underline{\boldsymbol{\epsilon}} (\eu)\|_{\Oh} \\
& \le C \|\esigma\|_{L^{2}(\mathcal{A},\Omega)} \| \underline{\boldsymbol{\epsilon}} (\eu)\|_{\Oh} + C \tau^{-\frac12} h^{-\frac12} \|\tau^{\frac12}(\Pim \eu - \euhat)\|_{\partial \Oh} \|\underline{\boldsymbol{\epsilon}} (\eu)\|_{\Oh} \\
& \quad + C h^t \|\underline{\boldsymbol{\sigma}}\|_{t, \Omega}\|\underline{\boldsymbol{\epsilon}} (\eu)\|_{\Oh} \qquad \text{by inverse inequality \eqref{classical_ineq_6}.}
\end{align*}
The proof is complete by the assumption $\tau = \mathcal{O}(h^{-1})$.
\end{proof}

Finally, combining Lemma \ref{estimate_gradu} and Theorem \ref{estimate_1}, after simple algebraic manipulation, we have our first error estimate:
\begin{corollary}\label{estimate_sigma}
Under the same assumption as in Theorem \ref{estimate_1}, we have
\begin{equation*}
\|\esigma\|_{L^{2}(\mathcal{A},\Omega)} + \|\tau^{\frac12}(\Pim \eu - \euhat)\|_{\partial \Oh} + \| \underline{\boldsymbol{\epsilon}} (\eu)\|_{\Oh} \le C( h^t \|\underline{\boldsymbol{\sigma}}\|_{t, \Omega} + h^{s-1} \|\boldsymbol{u}\|_{s, \Omega}),
\end{equation*}
for all $1 \le t \le k+1, 1 \le s \le k+2$, the constant $C$ is independent of $h$ and exact solution.
\end{corollary}

One can see that by taking $t=k+1, s=k+2$, both of the error $\esigma,
\underline{\boldsymbol{\epsilon}}(\eu)$ obtain optimal convergence
rate. Moreover, if we take $\tau=1/h$, we readily obtain the superconvergence
property
\[
 \|h^{\frac12}(\Pim \eu - \euhat)\|_{\partial \Oh}\le C\, h^{k+2},
\]
for smooth solutions. It is this superconvergence property the one which allows
to obtain the optimal convergence in the stress and, as we are going to see next, in the displacement.

{\textbf{Step 3: Estimate of $\eu$}.} Next we use a standard duality argument to get an estimate for $\eu$. First we present an important identity.
\begin{proposition}\label{duality_argument}
Assume that $(\boldsymbol{\phi}, \underline{\boldsymbol{\psi}}) \in \boldsymbol{H}^2(\Omega) \times \underline{\boldsymbol{H}}^1(\Omega)$ is the solution of the adjoint problem \eqref{dual_problem_1}, we have
\begin{alignat*}{1}
\|\eu\|^2_{\Omega} &= \bint{\mathcal{A} \esigma}{\underline{\boldsymbol{\psi}} - \Piv \underline{\boldsymbol{\psi}}} - \bint{\mathcal{A} (\deltasigma)}{\Piv \underline{\boldsymbol{\psi}}}\\
& \quad - \bintEh{\esigma \n - (\underline{\boldsymbol{\sigma}} \n - \widehat{\underline{\boldsymbol{\sigma}}}_h \n)}{\boldsymbol{\phi} - \Piw \boldsymbol{\phi}} + \bintEh{\eu - \euhat}{(\underline{\boldsymbol{\psi}} - \Piv \underline{\boldsymbol{\psi}})\n}.
\end{alignat*}
\end{proposition}

\begin{proof}
By the dual equation \eqref{dual_problem}, we can write
\begin{alignat*}{1}
\|\eu\|^2_{\Omega} & = \bint{\eu}{\nabla \cdot \underline{\boldsymbol{\psi}}} + \bint{\esigma}{\mathcal{A} \underline{\boldsymbol{\psi}} - \underline{\boldsymbol{\epsilon}} (\boldsymbol{\phi})} \\
& = \bint{\eu}{\nabla \cdot \underline{\boldsymbol{\psi}}} + \bint{\mathcal{A} \esigma}{\underline{\boldsymbol{\psi}}} - \bint{\esigma}{\nabla \boldsymbol{\phi}} \\
& = \bint{\eu}{\nabla \cdot \Piv \underline{\boldsymbol{\psi}}} + \bint{\mathcal{A} \esigma}{\Piv \underline{\boldsymbol{\psi}}} - \bint{\esigma}{\nabla \boldsymbol{\Piw \phi}}  \\
&+ \bint{\mathcal{A} \esigma}{\underline{\boldsymbol{\psi}} - \Piv \underline{\boldsymbol{\psi}}}+ \bint{\eu}{\nabla \cdot (\underline{\boldsymbol{\psi}} - \Piv \underline{\boldsymbol{\psi}})}  -  \bint{\esigma}{\nabla (\boldsymbol{\phi} - \Piw \boldsymbol{\phi})},
\intertext{integrating by parts for the last two terms, applying the property of the $L^2$-projections, yields,}
\|\eu\|^2_{\Omega}& = \bint{\eu}{\nabla \cdot \Piv \underline{\boldsymbol{\psi}}} + \bint{\mathcal{A} \esigma}{\Piv \underline{\boldsymbol{\psi}}} - \bint{\esigma}{\nabla \boldsymbol{\Piw \phi}}  \\
&+ \bint{\mathcal{A} \esigma}{\underline{\boldsymbol{\psi}} - \Piv \underline{\boldsymbol{\psi}}}+ \bintEh{\eu}{ (\underline{\boldsymbol{\psi}} - \Piv \underline{\boldsymbol{\psi}})\n}  -  \bintEh{\esigma \n}{\boldsymbol{\phi} - \Piw \boldsymbol{\phi}}.
\intertext{Taking $\underline{\boldsymbol{v}} := \Piv
  \underline{\boldsymbol{\psi}}$ and  $\boldsymbol{\omega} := \Piw
  \boldsymbol{\phi}$ in the error equations \eqref{error_equation_1} and 
\eqref{error_equation_2}, respectively, inserting these two equations into above identity, we obtain}
\|\eu\|^2_{\Omega}& = \bintEh{\euhat}{\Piv
  \underline{\boldsymbol{\psi}} \n} - \bint{\mathcal{A} (\deltasigma)}{\Piv
  \underline{\boldsymbol{\psi}}} - \bintEh{\underline{\boldsymbol{\sigma}}\n - \widehat{\underline{\boldsymbol{\sigma}}}_h \n}{ \boldsymbol{\Piw \phi}}  \\
&+ \bint{\mathcal{A} \esigma}{\underline{\boldsymbol{\psi}} - \Piv \underline{\boldsymbol{\psi}}}+ \bintEh{\eu}{ (\underline{\boldsymbol{\psi}} - \Piv \underline{\boldsymbol{\psi}})\n}  -  \bintEh{\esigma \n}{\boldsymbol{\phi} - \Piw \boldsymbol{\phi}}.
\end{alignat*}
Next, note that by the regularity assumption, $(\underline{\boldsymbol{\psi}}, \boldsymbol{\phi}) \in \underline{\boldsymbol{H}}^2(\Omega) \times \boldsymbol{H}^1(\Omega)$, so the normal component of $\underline{\boldsymbol{\psi}}$ and $\boldsymbol{\phi}$ are continuous across each face $F \in \Eh$. By the equation \eqref{strong_cont}, the normal component of $\widehat{\underline{\boldsymbol{\sigma}}}_h$ is also strongly continuous across each face $F \in \Eh$. This implies that
\begin{alignat*}{2}
-\bintEh{\euhat}{\underline{\boldsymbol{\psi}} \n} & =  -\bintEhb{\euhat}{\underline{\boldsymbol{\psi}} \n} = 0, \quad &&\text{by \eqref{error_equation_4}},  \\
\bintEh{\underline{\boldsymbol{\sigma}}\n - \widehat{\underline{\boldsymbol{\sigma}}}_h \n}{\boldsymbol{\phi}} & = \bintEhb{\underline{\boldsymbol{\sigma}}\n - \widehat{\underline{\boldsymbol{\sigma}}}_h \n}{\boldsymbol{\phi}}= 0 \quad
&&\text{by \eqref{dual_problem_3}}.
\end{alignat*}
Adding these two zero terms into the previous equation, rearranging the terms, we obtain the expression as presented in the proposition.
\end{proof}

As a consequence of the result just proved, 
we can obtain our estimate of $\eu$.
\begin{corollary}\label{estimate_u}
Under the same assumption as in Theorem \ref{estimate_1}, in addition, if the elliptic regularity property \eqref{regularity} holds,  then we have
\[
\|\eu\|_{\Omega} \le C( h^{t+1} \|\underline{\boldsymbol{\sigma}}\|_{t, \Omega} + h^{s} \|\boldsymbol{u}\|_{s, \Omega}),
\]
for $1 \le t \le k+1, 1 \le s \le k+2$.
\end{corollary}
\begin{proof}
We will estimate each of the terms on the right hand side of the identity in Proposition \ref{duality_argument}. 
\begin{align*}
\bint{\mathcal{A} \esigma}{\underline{\boldsymbol{\psi}} - \Piv \underline{\boldsymbol{\psi}}} \le C h \|\esigma\|_{L^{2}(\mathcal{A},\Omega)}
 \|\underline{\boldsymbol{\psi}}\|_{1, \Omega} \le C h \|\esigma\|_{L^{2}(\mathcal{A},\Omega)} \|\eu\|_{\Omega},
\end{align*}
by the projection property \eqref{classical_ineq_2} and the regularity assumption \eqref{regularity}.
\begin{align*}
\bint{\mathcal{A} (\deltasigma)}{\Piv \underline{\boldsymbol{\psi}}} &= \bint{\mathcal{A} (\deltasigma)}{\underline{\boldsymbol{\psi}}} - \bint{\mathcal{A} (\deltasigma)}{ \underline{\boldsymbol{\psi}} - \Piv \underline{\boldsymbol{\psi}}} \\
& \hspace{-0.5cm}= \bint{\deltasigma}{\mathcal{A}\underline{\boldsymbol{\psi}} - \overline{\mathcal{A}\underline{\boldsymbol{\psi}}}} - \bint{\mathcal{A} (\deltasigma)}{ \underline{\boldsymbol{\psi}} - \Piv \underline{\boldsymbol{\psi}}} \\
& \le C h \|\deltasigma\|_{\Omega} \|\underline{\boldsymbol{\psi}}\|_{1, \Omega}\\
&\le C h \|\deltasigma\|_{\Omega} \|\eu\|_{\Omega} \\
&\le C h^{t+1} \|\underline{\boldsymbol{\sigma}}\|_{t, \Omega} \|\eu\|_{\Omega},
\end{align*}
for all $0 \le t \le k+1$. Here we applied the Galerkin orthogonal property of the local $L^2$-projection $\Piv$ and the regularity assumption \eqref{regularity}.

For the third term, by the definition of the numerical trace \eqref{numerical_trace}, we have
\begin{align*}
\bintEh{\esigma \n - (\underline{\boldsymbol{\sigma}} \n - \widehat{\underline{\boldsymbol{\sigma}}}_h \n)}{\boldsymbol{\phi} - \Piw \boldsymbol{\phi}} & = - \bintEh{(\underline{\boldsymbol{\sigma}} - \Piv \underline{\boldsymbol{\sigma}}) \n)}{\boldsymbol{\phi} - \Piw \boldsymbol{\phi}} \\
& \quad + \bintEh{\widehat{\underline{\boldsymbol{\sigma}}}_h \n - \underline{\boldsymbol{\sigma}}_h \n}{\boldsymbol{\phi} - \Piw \boldsymbol{\phi}} \\
& = - \bintEh{(\underline{\boldsymbol{\sigma}} - \Piv \underline{\boldsymbol{\sigma}}) \n)}{\boldsymbol{\phi} - \Piw \boldsymbol{\phi}} \\
& \quad - \bintEh{\tau (\Pim \boldsymbol{u}_h - \widehat{\boldsymbol{u}}_h)}{\boldsymbol{\phi} - \Piw \boldsymbol{\phi}} \\
& = - \bintEh{(\underline{\boldsymbol{\sigma}} - \Piv \underline{\boldsymbol{\sigma}}) \n)}{\boldsymbol{\phi} - \Piw \boldsymbol{\phi}} \\
& \quad + \bintEh{\tau (\Pim \eu - \euhat)}{\boldsymbol{\phi} - \Piw \boldsymbol{\phi}} \\ 
& \quad + \bintEh{\tau (\Pim (\boldsymbol{u} - \Piw \boldsymbol{u})}{\boldsymbol{\phi} - \Piw \boldsymbol{\phi}}.
\end{align*}
To bound the third term, we only need to bound the above three terms individually.
\begin{align*}
\bintEh{(\underline{\boldsymbol{\sigma}} - \Piv \underline{\boldsymbol{\sigma}}) \n)}{\boldsymbol{\phi} - \Piw \boldsymbol{\phi}} & \le \|(\underline{\boldsymbol{\sigma}} - \Piv \underline{\boldsymbol{\sigma}}) \n)\|_{\partial \Oh}\|\boldsymbol{\phi} - \Piw \boldsymbol{\phi}\|_{\partial \Oh} \\
& \le C h^{t-\frac12} \|\underline{\boldsymbol{\sigma}}\|_{t, \Omega} h^{\frac32}\|\boldsymbol{\phi}\|_{2, \Omega},
\intertext{by the standard inequalities , \eqref{classical_ineq_4}, \eqref{classical_ineq_5},}
& \le C h^{t+1}\|\underline{\boldsymbol{\sigma}}\|_{t, \Omega} \|\eu\|_{\Omega},
\end{align*}
for all $1 \le t \le k+1$. The last step is due to the regularity assumption \eqref{regularity}.

Similarly, we apply the Cauchy-Schwarz inequality and \eqref{classical_ineq_4} for the other two terms:
\begin{align*}
 \bintEh{\tau (\Pim \eu - \euhat)}{\boldsymbol{\phi} - \Piw \boldsymbol{\phi}} & \le C \tau^{\frac12} \|\tau^{\frac12} (\Pim \eu - \euhat)\|_{\partial \Oh} \|\boldsymbol{\phi} - \Piw \boldsymbol{\phi}\|_{\partial \Oh} \\
& \le C \tau^{\frac12} h^{\frac32}  \|\tau^{\frac12} (\Pim \eu - \euhat)\|_{\partial \Oh} \|\boldsymbol{\phi}\|_{2, \Omega} \\
& \le C  \tau^{\frac12} h^{\frac32}  \|\tau^{\frac12} (\Pim \eu - \euhat)\|_{\partial \Oh} \|\eu\|_{\Omega}.
\end{align*}
\begin{align*}
\bintEh{\tau (\Pim (\boldsymbol{u} - \Piw \boldsymbol{u})}{\boldsymbol{\phi} - \Piw \boldsymbol{\phi}} & \le \tau \|\Pim (\boldsymbol{u} - \Piw \boldsymbol{u})\|_{\partial \Oh} \|\boldsymbol{\phi} - \Piw \boldsymbol{\phi}\|_{\partial \Oh} \\ 
& \le  \tau \|\boldsymbol{u} - \Piw \boldsymbol{u}\|_{\partial \Oh} \|\boldsymbol{\phi} - \Piw \boldsymbol{\phi}\|_{\partial \Oh} \\
& \le C \tau h^{s-\frac12} \|\boldsymbol{u}\|_{s, \Omega} h^{\frac32}\|\boldsymbol{\phi}\|_{2, \Omega} \\
& \le C \tau h^{s+1} \|\boldsymbol{u}\|_{s, \Omega}\|\eu\|_{\Omega},
\end{align*}
for all $1 \le s \le k+2$.

Finally, for the last term in Proposition \ref{duality_argument}, we can write:
\begin{align*}
\bintEh{\eu - \euhat}{(\underline{\boldsymbol{\psi}} - \Piv \underline{\boldsymbol{\psi}}) \n} & = \bintEh{\Pim \eu - \euhat}{(\underline{\boldsymbol{\psi}} - \Piv \underline{\boldsymbol{\psi}})\n} \\
&\quad + \bintEh{\eu - \Pim \eu}{(\underline{\boldsymbol{\psi}} - \Piv \underline{\boldsymbol{\psi}}) \n}.
\end{align*}
For the first term, we can apply a similar argument as in the previous steps to obtain:
\[
\bintEh{\Pim \eu - \euhat}{(\underline{\boldsymbol{\psi}} - \Piv \underline{\boldsymbol{\psi}})\n} \le C \tau^{-\frac12} h^{\frac12} \|\tau^{\frac12}(\Pim \eu - \euhat)\|_{\partial \Oh} \|\eu\|_{\Omega}.
\]
For the second term, we apply the same argument for the estimate of $T_{12}$ in the proof of Lemma \ref{Two_terms} and obtain:
\begin{align*}
\bintEh{\eu - \Pim \eu}{(\underline{\boldsymbol{\psi}} - \Piv \underline{\boldsymbol{\psi}}) \n} &\le C h \|\underline{\boldsymbol{\psi}}\|_{1, \Omega} \|\underline{\boldsymbol{\epsilon}}(\eu)\|_{\Oh} \\
& \le C h \|\underline{\boldsymbol{\epsilon}}(\eu)\|_{\Oh} \|\eu\|_{\Omega}.
\end{align*}
Finally if we take $\tau = \mathcal{O}(h^{-1})$, combine all the above estimates and Theorem \ref{estimate_sigma}, we obtain the estimate in Theorem \ref{estimate_u}. 
\end{proof}

As a consequence of Theorem \ref{estimate_sigma}, Theorem \ref{estimate_u}, we can obtain Theorem \ref{Main_result} by a simple triangle inequality and the approximation property of the projections $\Piw, \Piv$ \eqref{classical_ineq_1}, \eqref{classical_ineq_2}.

\bigskip

{\bf Step 4: Proof of locking-free result.}
We can now give the proof of Theorem~\ref{thm_locking_free}.
\begin{proof}[Proof of Theorem~\ref{thm_locking_free}]
In what follows, we assume that $s$ is some arbitrary real number in $[1, k+1]$ and $C$ is a positive constant 
independent of $P_{T}$ and $s$. We recall that $\esigma = \Piv \underline{\boldsymbol{\sigma}} 
- \underline{\boldsymbol{\sigma}}_h, \eu = \Piw \boldsymbol{u} - \boldsymbol{u}_h, \euhat  = 
\Pim \boldsymbol{u} - \widehat{\boldsymbol{u}}_h$.

For any $B\in \mathbb{R}^{3\times 3}$, we denote $B^{D}:=B - \frac{1}{3}\text{tr}B \, I_{3}$.
So, we have 
\begin{align*}
\esigma = \esigma^{D} + \frac{1}{3}\text{tr}\esigma I_{3}.
\end{align*}

By the Assumption~\ref{ass_isotropic} and Theorem~\ref{Main_result}, we have 
\begin{equation}
\label{esigmaD}
\|P_D^{\frac12}\esigma^{D}\Vert_{L^{2}(\Omega)}\leq \|\esigma \|_{L^{2}(\mathcal{A}, \Omega)}
\leq C h^{s} (\|\boldsymbol{u}\|_{s+1, \Omega} + \|\underline{\boldsymbol{\sigma}}\|_{s, \Omega}).
\end{equation}

In order to bound $\Vert \text{tr}\esigma\Vert_{L^{2}(\Omega)}$ independently of $P_{T}^{-1}$, we would like to use the 
well-known result \cite{BrezziFortin91, BBF-book} that for any $q\in L^{2}(\Omega)$ with $\int_{\Omega}q dx = 0$, we have 
\begin{align}
\label{brezzifortin_ineq}
\Vert q\Vert_{L^{2}(\Omega)} \leq C_{0} \sup_{\boldsymbol{\eta}\in \boldsymbol{H}_{0}^{1}(\Omega)} 
\dfrac{(q,\nabla\cdot\boldsymbol{\eta})_{\Omega}}{\Vert \boldsymbol{\eta}\Vert_{H^{1}(\Omega)}},
\end{align} 
for $C_0$ solely depends on the domain $\Omega$. By the assumption $\boldsymbol{g}=0$, taking $\underline{\boldsymbol{v}} = I_{3}$ 
in (\ref{error_equation_1}), we have that 
$\int_{\Omega} \text{tr}(\mathcal{A}\esigma)dx = 0$. 
According to the Assumption~\ref{ass_isotropic} and the fact that $P_{T}>0$, we have 
$\int_{\Omega} \text{tr}\esigma dx = 0$. 

For any $\boldsymbol{\eta}\in \boldsymbol{H}_{0}^{1}(\Omega)$, we have 
\begin{align*}
(\frac{1}{3}\text{tr}\esigma, \nabla\cdot\boldsymbol{\eta})_{\Omega} 
= & -(\nabla (\frac{1}{3}\text{tr}\esigma), \boldsymbol{\eta})_{\mathcal{T}_{h}} 
+ \langle (\frac{1}{3}\text{tr}\esigma)\n, \boldsymbol{\eta}\rangle_{\partial\mathcal{T}_{h}}\\
= & -(\nabla (\frac{1}{3}\text{tr}\esigma), \Piw\boldsymbol{\eta})_{\mathcal{T}_{h}} 
+ \langle (\frac{1}{3}\text{tr}\esigma)\n, \boldsymbol{\eta}\rangle_{\partial \mathcal{T}_{h}}\\ 
= & (\frac{1}{3}\text{tr}\esigma, \nabla\cdot \Piw\boldsymbol{\eta})_{\mathcal{T}_{h}} 
-\langle (\frac{1}{3}\text{tr}\esigma)\n, \Piw \boldsymbol{\eta} - \boldsymbol{\eta} 
\rangle_{\partial \mathcal{T}_{h}}\\
= & (\esigma - \esigma^{D}, \nabla\Piw\boldsymbol{\eta})_{\mathcal{T}_{h}} - \langle (\esigma - \esigma^{D})\n, 
\Piw\boldsymbol{\eta} - \boldsymbol{\eta} \rangle_{\partial\mathcal{T}_{h}}.
\end{align*}
By (\ref{error_equation_2}) with $\boldsymbol{\omega} = \Piw \boldsymbol{\eta}$, we have 
\begin{align*}
& (\frac{1}{3}\text{tr}\esigma, \nabla\cdot\boldsymbol{\eta})_{\Omega} \\
= & \bintEh{\underline{\boldsymbol{\sigma}} \n - \widehat{\underline{\boldsymbol{\sigma}}}_h \n}{\Piw \boldsymbol{\eta}}
-(\esigma^{D}, \nabla\Piw\boldsymbol{\eta})_{\mathcal{T}_{h}} - \langle (\esigma - \esigma^{D})\n, 
\Piw\boldsymbol{\eta} - \boldsymbol{\eta} \rangle_{\partial\mathcal{T}_{h}}\\
= & T_{1} + T_{2},
\end{align*}
where 
\begin{align*}
T_{1} := & \bintEh{\underline{\boldsymbol{\sigma}} \n - \widehat{\underline{\boldsymbol{\sigma}}}_h \n}{\Piw \boldsymbol{\eta}} 
- \langle \esigma \n, \Piw\boldsymbol{\eta} - \boldsymbol{\eta} \rangle_{\partial\mathcal{T}_{h}},\\
T_{2} := & -(\esigma^{D}, \nabla\Piw\boldsymbol{\eta})_{\mathcal{T}_{h}} + \langle \esigma^{D}\n, 
\Piw\boldsymbol{\eta} - \boldsymbol{\eta} \rangle_{\partial\mathcal{T}_{h}}.
\end{align*}

For the bound of $T_{1}$, by (\ref{strong_cont}) and the fact that $\boldsymbol{\eta}=0$ on $\partial\Omega$, we have 
\begin{align*}
& \bintEh{\underline{\boldsymbol{\sigma}} \n - \widehat{\underline{\boldsymbol{\sigma}}}_h \n}{\Piw \boldsymbol{\eta}} \\ 
= & \bintEh{\underline{\boldsymbol{\sigma}} \n}{\Piw \boldsymbol{\eta} - \boldsymbol{\eta}} 
-\bintEh{\widehat{\underline{\boldsymbol{\sigma}}}_h \n}
{\Piw \boldsymbol{\eta} - \Pim\boldsymbol{\eta}} \\
= & \bintEh{\underline{\boldsymbol{\sigma}} \n - \widehat{\underline{\boldsymbol{\sigma}}}_h \n}
{\Piw \boldsymbol{\eta} - \boldsymbol{\eta}} \\
= & \langle (\underline{\boldsymbol{\sigma}} - \Piv \underline{\boldsymbol{\sigma}})\n, 
\Piw\boldsymbol{\eta} - \boldsymbol{\eta} \rangle_{\partial\mathcal{T}_{h}} 
+\langle \esigma\n, 
\Piw\boldsymbol{\eta} - \boldsymbol{\eta} \rangle_{\partial\mathcal{T}_{h}} 
+\langle \tau(\Pim \boldsymbol{u}_{h} - \widehat{\boldsymbol{u}}_{h}), 
\Piw\boldsymbol{\eta} - \boldsymbol{\eta} \rangle_{\partial\mathcal{T}_{h}}. 
\end{align*}
So, we have 
\begin{align}
\label{T1_final_form_locking_free}
T_{1} = \langle (\underline{\boldsymbol{\sigma}} - \Piv \underline{\boldsymbol{\sigma}})\n, 
\Piw\boldsymbol{\eta} - \boldsymbol{\eta} \rangle_{\partial\mathcal{T}_{h}} 
+\langle \tau(\Pim \boldsymbol{u}_{h} - \widehat{\boldsymbol{u}}_{h}), 
\Piw\boldsymbol{\eta} - \boldsymbol{\eta} \rangle_{\partial\mathcal{T}_{h}}.  
\end{align}
According to Corollary~\ref{estimate_sigma}, we have 
\begin{equation}\label{trace_error}
 \|\tau^{\frac12}(\Pim \eu - \euhat)\|_{\partial \Oh} \leq 
 C h^{s} (\|\boldsymbol{u}\|_{s+1, \Omega} + \|\underline{\boldsymbol{\sigma}}\|_{s, \Omega}).
\end{equation}
By the definition of $\euhat$ and $\eu$, we have 
\begin{align*}
 \|\tau^{\frac12}(\Pim \boldsymbol{u}_{h} - \widehat{\boldsymbol{u}}_{h})\|^2_{\partial \Oh} 
 & \leq 2\|\tau^{\frac12}( \Pim \eu - \euhat)\|^2_{\partial \mathcal{T}_h} 
 +2\|\tau^{\frac12} \Pim (\boldsymbol{u} - \Piw \boldsymbol{u})\|^2_{\partial \Oh} \\
& \le 2\|\tau^{\frac12}( \Pim \eu - \euhat)\|^2_{\partial \mathcal{T}_h} 
+2 \|\tau^{\frac12} (\boldsymbol{u} - \Piw \boldsymbol{u})\|^2_{\partial \Oh}.
\end{align*}
Now applying Young's inequality and \eqref{trace_error}, \eqref{classical_ineq_1} we obtain:
\begin{equation}
\label{jump_difference_ineq1}
 \|\tau^{\frac12}(\Pim \boldsymbol{u}_{h} - \widehat{\boldsymbol{u}}_{h})\|_{\partial \Oh} \leq 
 C h^{s} (\|\boldsymbol{u}\|_{s+1, \Omega} + \|\underline{\boldsymbol{\sigma}}\|_{s, \Omega}).
\end{equation}
According to \eqref{T1_final_form_locking_free}, \eqref{jump_difference_ineq1}, we have 
\begin{align}
\label{T1_bound_locking_free}
T_{1} \leq  C h^{s} (\|\boldsymbol{u}\|_{s+1, \Omega} + \|\underline{\boldsymbol{\sigma}}\|_{s, \Omega})
\Vert \boldsymbol{\eta}\Vert_{H^{1}(\Omega)}.
\end{align}  

For the bound of $T_{2}$, we have 
\begin{align*}
T_{2} = & -(\esigma^{D}, \nabla\Piw\boldsymbol{\eta})_{\mathcal{T}_{h}} + \langle \esigma^{D}\n, 
\Piw\boldsymbol{\eta} - \boldsymbol{\eta} \rangle_{\partial\mathcal{T}_{h}}\\
= & -(\esigma^{D}, \nabla(\Piw\boldsymbol{\eta}-\boldsymbol{\eta}))_{\mathcal{T}_{h}} 
+ \langle \esigma^{D}\n, \Piw\boldsymbol{\eta} - \boldsymbol{\eta} \rangle_{\partial\mathcal{T}_{h}}
-(\esigma^{D}, \nabla \boldsymbol{\eta})_{\mathcal{T}_{h}}\\
= & (\nabla\cdot\esigma^{D}, \Piw\boldsymbol{\eta}-\boldsymbol{\eta})_{\mathcal{T}_{h}} 
-(\esigma^{D}, \nabla \boldsymbol{\eta})_{\mathcal{T}_{h}}\\ 
= & -(\esigma^{D}, \nabla \boldsymbol{\eta})_{\mathcal{T}_{h}}.
\end{align*}
By (\ref{esigmaD}), we have 
\begin{align}
\label{T2_bound_locking_free}
T_{2} \leq C h^{s} (\|\boldsymbol{u}\|_{s+1, \Omega} + \|\underline{\boldsymbol{\sigma}}\|_{s, \Omega}) 
\Vert \boldsymbol{\eta}\Vert_{H^{1}(\Omega)}.
\end{align}

Finally, combining the estimates \eqref{esigmaD}, \eqref{brezzifortin_ineq}, \eqref{T1_bound_locking_free}, 
\eqref{T2_bound_locking_free}, we have 
\begin{align*}
\Vert \esigma\Vert_{L^{2}(\Omega)} \leq C_{1} h^{s} (\|\boldsymbol{u}\|_{s+1, \Omega} 
+ \|\underline{\boldsymbol{\sigma}}\|_{s, \Omega}) 
\Vert \boldsymbol{\eta}\Vert_{H^{1}(\Omega)}.
\end{align*}
Here the constant $C_{1}$ is independent of $P_{T}^{-1}$.
\end{proof}

\section{Numerical Experiment} \label{numerical}
In this section,  we display numerical experiments in 2D to verify the error estimates provided in Theorem \ref{Main_result}.  We also display numerical results showing that our method does not exhibit volumetric-locking when the material tends to be incompressible.  In addition, our numerical results suggest that the error estimates provided in Theorem \ref{thm_locking_free} for the incompressible limit case are \textit{sharp}.

We carry out the numerical experiments on the domain $\Omega = (0,1) \times (0,1)$ and monitor the errors $\| \Piv \underline{\boldsymbol{\sigma}} - \underline{\boldsymbol{\sigma}}_h \Vert_{L^{2}(\Omega)}$ and $\| \Piw \boldsymbol{u} - \boldsymbol{u}_h\Vert_{L^{2}(\Omega)}$.To explore the dependence of  the convergence properties of our method with respect to the form of the meshes, we consider two types of meshes, as shown in FIGURE \ref{fig:Mesh}.
\begin{figure}
\includegraphics[scale=0.3]{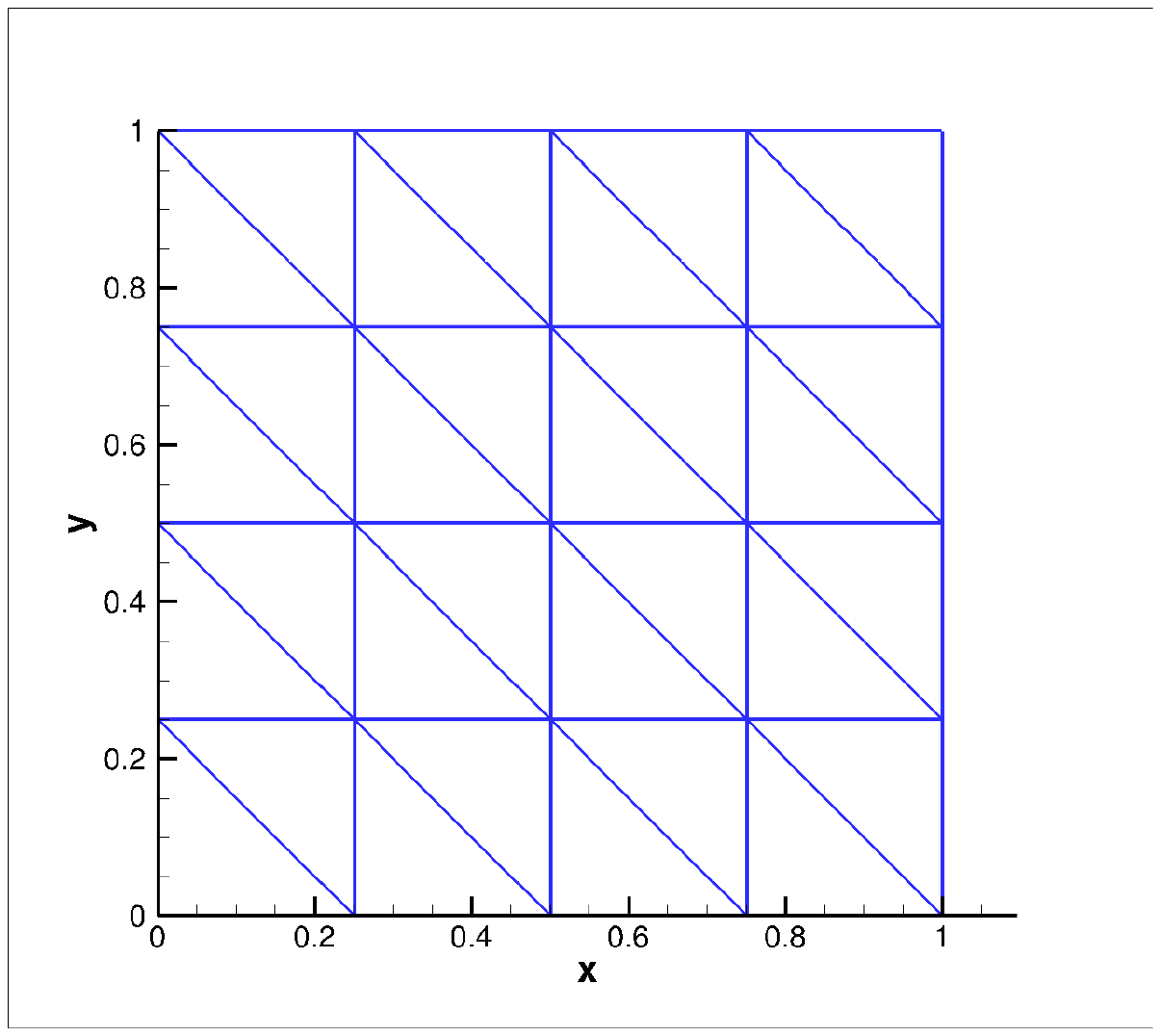}
\includegraphics[scale=0.3]{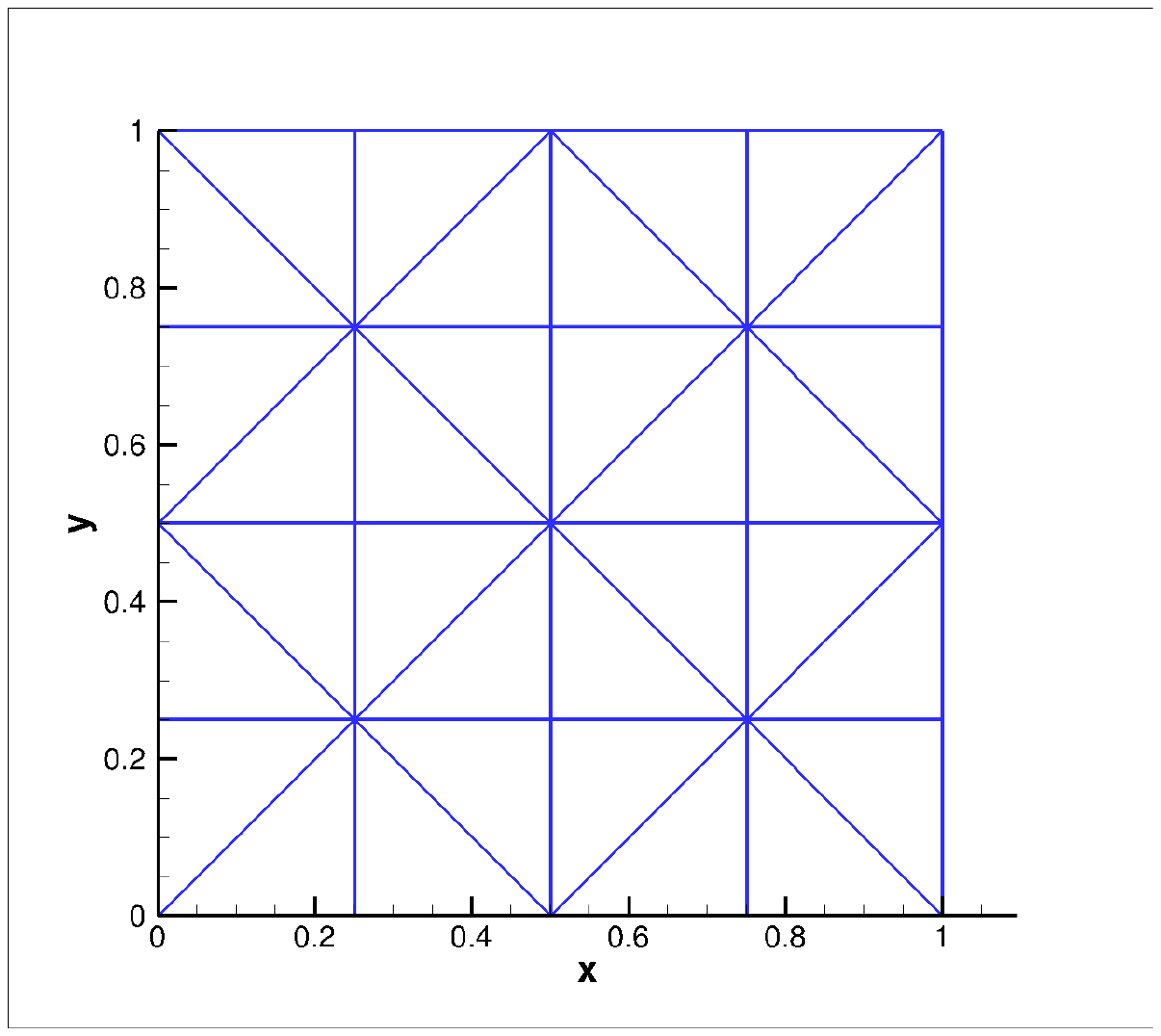} \caption{An example of Mesh-$1$(left) and Mesh-$2$(right) with $h = 0.354$}
\label{fig:Mesh} 
\end{figure}

\begin{table}[h]
\begin{center}
\scriptsize
\begin{tabular}{  c  c  c  c  c  c  c  c  c  c } \hline
 & &\multicolumn{2}{c}{$\| \Piv \underline{\boldsymbol{\sigma}} - \underline{\boldsymbol{\sigma}}_h \Vert_{L^{2}(\Omega)}$} & \multicolumn{2}{c}{$\|\Piw \boldsymbol{u} - \boldsymbol{u}_h\Vert_{L^{2}(\Omega)}$} & \multicolumn{2}{c}{$\| \Piv \underline{\boldsymbol{\sigma}} - \underline{\boldsymbol{\sigma}}_h \Vert_{L^{2}(\Omega)}$} & \multicolumn{2}{c}{$\| \Piw \boldsymbol{u} - \boldsymbol{u}_h\Vert_{L^{2}(\Omega)}$}  \\\hline
  & &\multicolumn{4}{c}{Mesh-$1$} & \multicolumn{4}{c}{Mesh-$2$}  \\\hline
$k$ & Mesh & Error & Order  & Error & Order  & Error & Order & Error & Order \\ \hline 
\multirow{3}{*}{0} & $h$ & 9.81E-02 & - & 3.74E-03 & - & 4.20E-01 & - &8.28E+12 & -\\ 
&$h/2$ & 9.50E-02 & 0.05 & 3.69E-02 & 0.02 & 2.14E-02 & 0.97 &1.05E+12 &2.98 \\
&$h/4$ & 9.42E-02 & 0.01 & 3.68E-03 & 0.00 & 1.05E-02 & 1.02    &1.41E+11&2.90 \\ 
&$h/8$ & 9.41E-02 & 0.00 & 3.68E-03 & 0.00 & 5.22E-03 & 1.01    &1.79E+10&2.98 \\
&$h/16$ & 9.41E-02 & 0.00 & 3.68E-03 & 0.00 & 8.17E-01 & -3.97    &1.39E+12&-6.28 \\
\hline
\multirow{3}{*}{1}& $h$ & 2.26E-03 & - & 1.88E-03 & - & 2.04E-03 & - & 9.41E-04 & -\\ 
&$h/2$ & 7.24E-03 & 1.65 & 3.57E-04 & 2.40 & 5.90E-03 & 1.79 & 1.45E-04 & 2.70 \\ 
&$h/4$ & 2.09E-03 & 1.79 & 5.51E-05 & 2.69 & 1.58E-03 & 1.92 & 2.00E-05 & 2.86\\ 
&$h/8$ & 5.60E-04 & 1.90 & 7.60E-06 & 2.86 & 4.08E-04 & 1.95 & 2.62E-06 & 2.93\\
&$h/16$ & 1.45E-04 & 1.95 & 9.93E-07 & 2.94 & 7.01E-06 & 2.00 & 3.35E-07 & 2.97\\
\hline
\multirow{3}{*}{2}& $h$ &1.24E-03  & - &5.52E-05  & - &1.23E-03  & - &3.53E-05  & -\\ 
&$h/2$ &1.57E-04  &2.98  &3.74E-06  &3.88  &1.57E-04  &2.97  &2.25E-06  &3.97 \\ 
&$h/4$ &1.97E-05  &2.99  &2.43E-07  &3.95  &1.97E-05  &2.99  &1.42E-07  &3.99 \\ 
&$h/8$ &2.46E-06  &3.00  &1.54E-08  &3.97  &2.47E-06  &3.00  &8.90E-09  &3.99 \\
&$h/16$ &3.08E-07  &3.00  &9.73E-10  &3.99 &3.10E-07  &3.00  &5.58E-10  &4.00 \\
\hline
\multirow{3}{*}{3}& $h$ &5.26E-05  & - &1.45E-06  & - &5.33E-05  & - &1.27E-06  & -\\ 
&$h/2$ &3.51E-06  &3.90  &4.90E-08  &4.89  &3.54E-06  &3.91  &4.36E-08  &4.86 \\ 
&$h/4$ &2.26E-07  &3.96  &1.59E-09  &4.95  &2.29E-07  &3.95  &1.43E-09  &4.93 \\ 
&$h/8$ &1.42E-08  &3.98  &5.12E-11  &4.96 &1.45E-08  &3.98  &4.59E-11  &4.96 \\
\hline
\end{tabular}
\end{center}
\caption{History of convergence for the exact solution (\ref{eq:Test_1}) where $h = 0.177$}\label{table:smooth}
\end{table}

\begin{figure}
\centering
\captionsetup{justification=centering}
\includegraphics[height=200mm, width = 135mm]{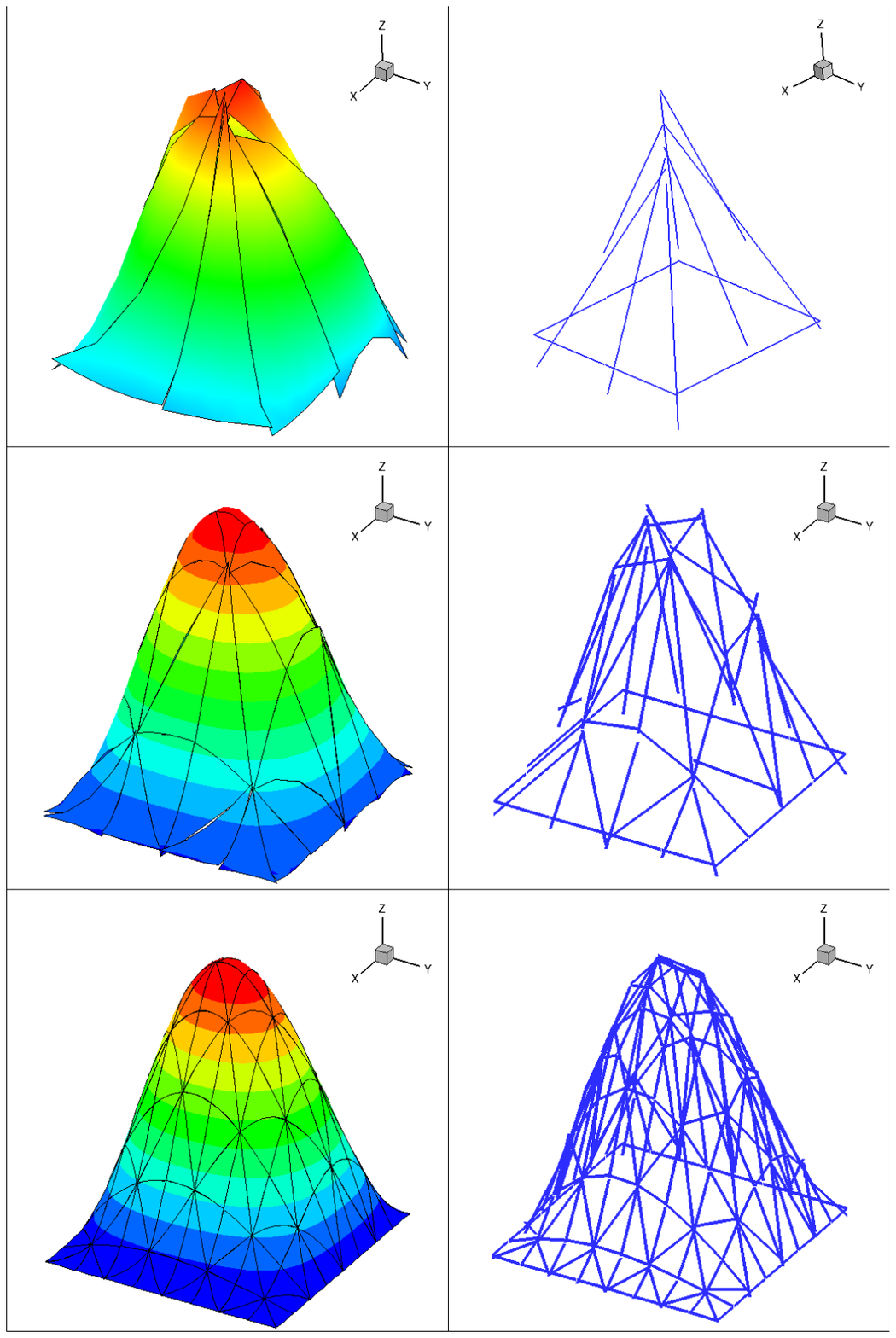} \caption{Convergence sequence of the displacement on Mesh-2 for $k=1$. Left: $u_h^1$(quadratic), Right: $\widehat{u}_h^1$(linear) }
\label{fig:convseq} 
\end{figure}

\subsection{Order of convergence of our HDG method}
In this section, we consider an isotropic material in 2D with plain stress condition and take the Poisson Ratio $\nu = 0.3$ and the Young's Modulus $E = 1$:
\begin{align}
 \mathcal{A}\underline{\boldsymbol{\sigma}} &= \frac{1+\nu}{E}\underline{\boldsymbol{\sigma}}
- \frac{\nu}{E}\text{tr}(\underline{\boldsymbol{\sigma}})I_{2}.
\end{align}
In particular, we test our HDG method on a smooth solution $\boldsymbol{u} = (u_1,u_2)$ in \cite{SoonCockburnStolarski09}, such that:
\begin{eqnarray} \label{eq:Test_1}
u_1 = 10 \sin(\pi x) (1-x)(y-y^2)(1-0.5y), \ \ \ \ u_2 = 0.
\end{eqnarray} 
We set $\boldsymbol{f}$ and $\boldsymbol{g}$ to satisfy the above exact solution (\ref{eq:Test_1}).  To explore the convergence properties of our method, we conduct numerical experiments for $k=0,1,2,3$ and take $\tau = \mathcal{O}(\frac{1}{h})$. The history of convergence is displayed in Table \ref{table:smooth}. We observe that when $k \geq 1$, our method converges with order $k+1$ in the stress and order $k+2$ in the displacement for both Mesh-$1$ and Mesh-$2$.  In addition, the numerical results suggest that our method does not converge to the exact solution when $k=0$.  To aid visualization, we also plot the convergence sequence of the displacement in FIGURE \ref{fig:convseq}. 

\begin{table}
\begin{center}
\scriptsize
\begin{tabular}{  c  c  c  c  c  c  c  c  c  c } \hline
 & &\multicolumn{2}{c}{$\| \Piv \underline{\boldsymbol{\sigma}} - \underline{\boldsymbol{\sigma}}_h \Vert_{L^{2}(\Omega)}$} & \multicolumn{2}{c}{$\|\Piw \boldsymbol{u} - \boldsymbol{u}_h\Vert_{L^{2}(\Omega)}$} & \multicolumn{2}{c}{$\| \Piv \underline{\boldsymbol{\sigma}} - \underline{\boldsymbol{\sigma}}_h \Vert_{L^{2}(\Omega)}$} & \multicolumn{2}{c}{$\| \Piw \boldsymbol{u} - \boldsymbol{u}_h\Vert_{L^{2}(\Omega)}$}  \\\hline
  & &\multicolumn{8}{c}{$\nu = 0.49$}   \\\hline
 & &\multicolumn{4}{c}{Mesh-$1$} & \multicolumn{4}{c}{Mesh-$2$}  \\\hline
$k$ & Mesh & Error & Order  & Error & Order  & Error & Order & Error & Order \\ \hline 
\multirow{3}{*}{1} & $h$ & 4.12E-03 & - & 1.14E-04 & - & 4.12E-03 & - &9.15E-04 & -\\ 
&$h/2$ & 1.22E-03 & 1.75 & 2.53E-05 & 2.17 & 1.27E-03 & 1.70 &1.47E-05 & 2.64\\ 
&$h/4$ & 3.32E-04 & 1.88 & 4.76E-06 & 2.41 & 3.40E-04 & 1.90 &2.00E-06 &2.87 \\
&$h/8$ & 8.69E-05 & 1.93 & 8.17E-07 & 2.54 & 8.64E-05 & 1.98 &2.58E-07&2.96 \\ 
&$h/16$ & 2.22E-05 & 1.97 & 1.23E-07 & 2.73 & 2.17E-05 & 1.99 &3.27E-08&2.98 \\
\hline
\multirow{3}{*}{2}& $h$ & 9.33E-04 & - & 2.00E-05 & - & 9.37E-04 & - & 1.24E-05 & -\\ 
&$h/2$ & 1.29E-04 & 2.85 & 1.65E-06 & 3.60 & 1.32E-04 & 2.83 & 9.42E-07 & 3.71 \\ 
&$h/4$ & 1.65E-05 & 2.97 & 1.17E-07 & 3.82 & 1.64E-05 & 3.00 & 6.01E-08 & 3.97\\ 
&$h/8$ & 2.07E-06 & 3.00 & 7.76E-09 & 3.92 & 2.05E-06 & 3.00 & 3.77E-09 & 3.99\\
&$h/16$ & 2.58E-07 & 3.00 & 4.98E-10 & 3.96 & 2.56E-07 & 3.00 & 2.36E-10 & 4.00\\
\hline
\multirow{3}{*}{3}& $h$ &1.44E-04  & - &1.65E-06  & - &1.57E-04  & - &1.53E-06  & -\\ 
&$h/2$ &9.78E-06  &3.88  &6.18E-08  &4.74  &9.87E-06  &3.99  &5.11E-08  &4.90 \\ 
&$h/4$ &6.27E-07  &3.96  &2.09E-09  &4.89  &6.19E-07  &3.99  &1.68E-09  &4.93 \\ 
&$h/8$ &3.95E-08  &3.99  &6.77E-11  &4.95  &3.89E-08  &3.99  &5.43E-11  &4.95 \\
&$h/16$ &2.49E-09  &3.99  &2.21E-12  &4.94 &2.44E-09  &4.00  &1.74E-12  &4.97 \\
\hline
  & &\multicolumn{8}{c}{$\nu = 0.4999$}   \\\hline
 & &\multicolumn{4}{c}{Mesh-$1$} & \multicolumn{4}{c}{Mesh-$2$}  \\\hline
$k$ & Mesh & Error & Order  & Error & Order  & Error & Order & Error & Order \\ \hline 
\multirow{3}{*}{1} & $h$ & 4.12E-03 & - & 1.13E-04 & - & 4.13E-03 & - &9.05E-04 & -\\ 
&$h/2$ & 1.22E-03 & 1.76 & 2.52E-05 & 2.17 & 1.26E-03 & 1.71 &1.45E-05 & 2.64\\ 
&$h/4$ & 3.31E-04 & 1.88 & 4.72E-06 & 2.41 & 3.39E-04 & 1.90 &1.98E-06 &2.87 \\
&$h/8$ & 8.66E-05 & 1.93 & 8.11E-07 & 2.54 & 8.61E-05 & 1.98 &2.55E-07&2.96 \\ 
&$h/16$ & 2.21E-05 & 1.97 & 1.22E-07 & 2.73 & 2.16E-05 & 1.99 &3.23E-08&2.98 \\
\hline
\multirow{3}{*}{2}& $h$ & 9.32E-04 & - & 1.98E-05 & - & 9.34E-04 & - & 1.22E-05 & -\\ 
&$h/2$ & 1.29E-04 & 2.86 & 1.64E-06 & 3.60 & 1.32E-04 & 2.83 & 9.31E-07 & 3.72 \\ 
&$h/4$ & 1.64E-05 & 2.97 & 1.16E-07 & 3.82 & 1.64E-05 & 3.00 & 5.94E-08 & 3.97\\ 
&$h/8$ & 2.06E-06 & 3.00 & 7.70E-09 & 3.92 & 2.04E-06 & 3.00 & 3.73E-09 & 3.99\\
&$h/16$ & 2.57E-07 & 3.00 & 4.95E-10 & 3.96 & 2.55E-07 & 3.00 & 2.33E-10 & 4.00\\
\hline
\multirow{3}{*}{3}& $h$ &1.44E-04  & - &1.63E-06  & - &1.57E-04  & - &1.51E-06  & -\\ 
&$h/2$ &9.75E-06  &3.88  &6.09E-08  &4.74  &9.83E-06  &3.99  &5.03E-08  &4.90 \\ 
&$h/4$ &6.25E-07  &3.96  &2.06E-09  &4.89  &6.17E-07  &3.99  &1.66E-09  &4.93 \\ 
&$h/8$ &3.94E-08  &3.99  &6.72E-11  &4.95  &3.87E-08  &3.99  &5.39E-11  &4.94 \\
&$h/16$ &2.48E-09  &3.99  &2.20E-12  &4.93 &2.43E-09  &3.99  &1.73E-12  &4.96 \\
\hline
  & &\multicolumn{8}{c}{$\nu = 0.49999$}   \\\hline
 & &\multicolumn{4}{c}{Mesh-$1$} & \multicolumn{4}{c}{Mesh-$2$}  \\\hline
$k$ & Mesh & Error & Order  & Error & Order  & Error & Order & Error & Order \\ \hline 
\multirow{3}{*}{1} & $h$ & 4.12E-03 & - & 1.13E-04 & - & 4.13E-03 & - &9.05E-04 & -\\ 
&$h/2$ & 1.22E-03 & 1.76 & 2.52E-05 & 2.17 & 1.26E-03 & 1.71 &1.45E-05 & 2.64\\ 
&$h/4$ & 3.31E-04 & 1.88 & 4.72E-06 & 2.41 & 3.39E-04 & 1.90 &1.98E-06 &2.87 \\
&$h/8$ & 8.66E-05 & 1.93 & 8.11E-07 & 2.54 & 8.61E-05 & 1.98 &2.55E-07&2.96 \\ 
&$h/16$ & 2.21E-05 & 1.97 & 1.22E-07 & 2.73 & 2.16E-05 & 1.99 &3.23E-08&2.98 \\
\hline
\multirow{3}{*}{2}& $h$ & 9.32E-04 & - & 1.98E-05 & - & 9.34E-04 & - & 1.22E-05 & -\\ 
&$h/2$ & 1.29E-04 & 2.86 & 1.64E-06 & 3.60 & 1.32E-04 & 2.83 & 9.31E-07 & 3.72 \\ 
&$h/4$ & 1.64E-05 & 2.97 & 1.16E-07 & 3.82 & 1.64E-05 & 3.00 & 5.94E-08 & 3.97\\ 
&$h/8$ & 2.06E-06 & 3.00 & 7.70E-09 & 3.92 & 2.04E-06 & 3.00 & 3.73E-09 & 3.99\\
&$h/16$ & 2.57E-07 & 3.00 & 4.95E-10 & 3.96 & 2.55E-07 & 3.00 & 2.33E-10 & 4.00\\
\hline
\multirow{3}{*}{3}& $h$ &1.44E-04  & - &1.63E-06  & - &1.57E-04  & - &1.50E-06  & -\\ 
&$h/2$ &9.75E-06  &3.88  &6.09E-08  &4.74  &9.83E-06  &3.99  &5.03E-08  &4.90 \\ 
&$h/4$ &6.25E-07  &3.96  &2.06E-09  &4.89  &6.17E-07  &3.99  &1.66E-09  &4.93 \\ 
&$h/8$ &3.94E-08  &3.99  &6.72E-11  &4.95  &3.86E-08  &3.99  &5.39E-11  &4.94 \\
&$h/16$ &2.48E-09  &3.99  &2.20E-12  &4.93 &2.44E-09  &3.98  &1.74E-12  &4.95 \\
\hline
\end{tabular}
\end{center}
\caption{History of convergence for the exact solution (\ref{eq:Test_2}) where $h = 0.354$}\label{table:lockingtest}
\end{table}

\begin{figure}
\centering
\captionsetup{justification=centering}
\includegraphics[height=200mm, width = 135mm]{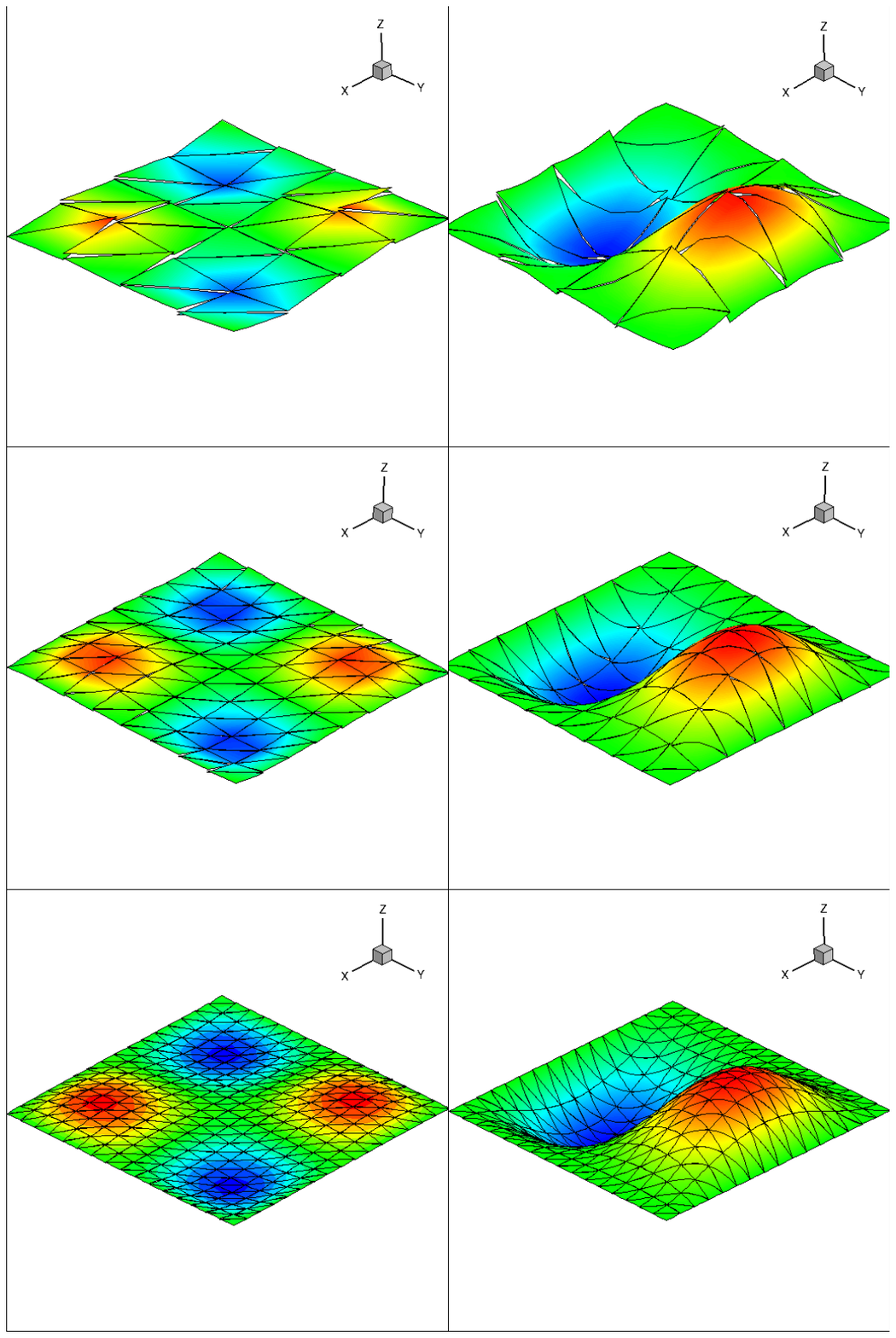} \caption{ Convergence sequence of the stress and the displacement on Mesh-1 for $k=1$ and $\nu = 0.49999$.  Left: $\sigma_h^{11}$(linear), right $u_h^1$(quadratic). }
\label{fig:convseq2} 
\end{figure}
\subsection{Locking experiments}
In this section, we consider an isotropic material in 2D with plane-strain condition:
\begin{align}
 \mathcal{A}\underline{\boldsymbol{\sigma}} &= \frac{1+\nu}{E}\underline{\boldsymbol{\sigma}}
- \frac{(1+\nu)\nu}{E}\text{tr}(\underline{\boldsymbol{\sigma}})I_{2}
\end{align}
where $\nu$ is the Poisson Ratio and $E$ is the Young's Modulus. This example satisfies the Assumption 2.1 with $P_D = \frac{1+\nu}{E}$ and $P_T = \frac{(1+\nu)}{E}(1 -2\nu)$. By sending $\nu \to 0.5$, this material is nearly incompressible. We consider an example in \cite{BercovierLivne79, SoonCockburnStolarski09} by setting $\boldsymbol{f}$ and $\boldsymbol{g}$ to satisfy the exact solution: \begin{eqnarray} \label{eq:Test_2}
u_1 &=& -x^2(x-1)^2y(y-1)(2y-1) \\
u_2 &=& \ \ y^2(y-1)^2x(x-1)(2x-1)
\end{eqnarray} 
with $E = 3$. We conduct numerical experiments for this problem for $k=1,2,3$ with $\tau = \mathcal{O}(\frac{1}{h})$. The history of convergence is displayed in Table \ref{table:lockingtest} and the convergence sequence of the stress and the displacement is plotted in FIGURE \ref{fig:convseq2}. By increasing $\nu$ from $0.49$ to $0.49999$, we observe the same order of convergence which is optimal in both stress and displacement. In addition, our numerical results demonstrate that the convergence properties of our method do not depend on the type of meshes.  Altoghether, this observation exactly aligns with the error estimates provided in Theorem 2.2 and it justifies that our HDG method is free from volumetric locking.

{\bf Acknowledgements}. 
The work of the first author was partially supported by a grant from the Research Grants Council of 
the Hong Kong Special Administrative Region, China (Project No. CityU 11302014). As a convention the names of the authors are alphabetically ordered. All authors contributed equally in this article. Finally authors would like to thank Guosheng Fu at University of Minnesota for fruitful discussion.






\providecommand{\bysame}{\leavevmode\hbox to3em{\hrulefill}\thinspace}
\providecommand{\MR}{\relax\ifhmode\unskip\space\fi MR }
\providecommand{\MRhref}[2]{%
  \href{http://www.ams.org/mathscinet-getitem?mr=#1}{#2}
}
\providecommand{\href}[2]{#2}

\end{document}